\documentclass[12pt]{amsart}
\usepackage{amssymb}
\usepackage{mathtools}
\usepackage{txfonts}
\usepackage{eucal}
\usepackage{extarrows}

\usepackage{color}

\usepackage[all]{xy}

\usepackage{tikz}

\usepackage{graphicx}
\usepackage{float}

\usepackage{xspace}

\usepackage[a4paper,body={16.3cm,22.8cm},centering]{geometry}
\usepackage[colorlinks,final,backref=page,hyperindex,pdftex]{hyperref}

\newcommand{\nc}{\newcommand}
\newcommand{\delete}[1]{}

\nc{\mlabel}[1]{\label{#1}}  
\nc{\mcite}[1]{\cite{#1}}  
\nc{\mref}[1]{\ref{#1}}  
\nc{\mbibitem}[1]{\bibitem{#1}} 

\delete{
\nc{\mlabel}[1]{\label{#1}  
{\hfill \hspace{1cm}{\small\tt{{\ }\hfill(#1)}}}}
\nc{\mcite}[1]{\cite{#1}{\small{\tt{{\ }(#1)}}}}  
\nc{\mref}[1]{\ref{#1}{{\tt{{\ }(#1)}}}}  
\nc{\mbibitem}[1]{\bibitem[\bf #1]{#1}} 
}

\newtheorem{theorem}{Theorem}[section]
\newtheorem{prop}[theorem]{Proposition}

\theoremstyle{definition}
\newtheorem{defn}[theorem]{Definition}
\newtheorem{remark}[theorem]{Remark}
\newtheorem{exam}[theorem]{Example}
\newtheorem{prop-def}{Proposition-Definition}[section]

\newcommand\cal[1]{\mathcal{#1}}

\newcommand\alphlist{a,b,c,d,e,f,g,h,i,j,k,l,m,n,o,p,q,r,s,t,u,v,w,x,y,z}
\newcommand\Alphlist{A,B,C,D,E,F,G,H,I,J,K,L,M,N,O,P,Q,R,S,T,U,V,W,X,Y,Z}
\newcommand\getcmds[3]{\expandafter\newcommand\csname #2#1\endcsname{#3{#1}}}
\makeatletter
\@for\x:=\alphlist\do{\expandafter\getcmds\expandafter{\x}{frak}{\mathfrak}}
\@for\x:=\Alphlist\do{\expandafter\getcmds\expandafter{\x}{frak}{\mathfrak}}
\makeatother

\nc{\bfk}{{\bf k}}
\font\cyr=wncyr10

\newfont{\scyr}{wncyr10 scaled 550}
\nc{\sha}{\mbox{\cyr X}}
\nc{\ssha}{\mbox{\bf \scyr X}}

\nc{\id}{\mathrm{id}}
\nc{\Id}{\mathrm{Id}}
\nc{\lbar}[1]{\overline{#1}}
\nc{\ot}{\otimes}
\nc{\dep}{\mathrm{dep}}
\nc{\leaf}{\mathrm{leaf}}

\usetikzlibrary{calc}
\newlength\xch
\newsavebox\dbox
\sbox\dbox{\tikz{\fill (0,0) circle (0.05cm);}}
\newif\ifqdd
\newif\ifzdd
\newcommand\cddf[3]{%
\coordinate (#2) at ($(#1)+(#3)$);
\draw (#1)--(#2);
\ifqdd\node at (#1) {\usebox\dbox};\fi
\ifzdd\node at (#2) {\usebox\dbox};\fi}
\newcommand\cdx[4][1]{\cddf{#2}{#3}{#4:#1*\xch}}
\newcommand\cdb[2][1]{\cdx[#1]{#2}{#2b}{-90}}
\let\treeoo\treeo%
\newcommand\ocdx[6][1]{%
\node[draw,circle,minimum size=2pt,label={#6:$#5$}]
(#3) at ($(#2)+(#4:#1*\xch)$) {};
\draw (#2)--(#3);}

\newcommand\scopeclip[1]{\begin{scope}
\clip(-1.1,-0.5)rectangle(1.1,1);#1\end{scope}}
\newcommand\XX[2][]{%
\tikz[line width=0.15ex,x=0.5cm,y=0.5cm,baseline,inner sep=1.5pt,
every node/.style={font=\scriptsize},#1]{
\scopeclip{\draw (135:1.5)--(0,0)--(45:1.5) (0,-0.5)--(0,0);}#2}}
\newcommand\xx[3]{%
\scopeclip{\draw(#1/10,#2/10)--+(#3*45:2.5);}}
\newcommand\xxl[2]{\xx{#1}{#2}3}
\newcommand\xxr[2]{\xx{#1}{#2}1}
\newcommand\xxlr[2]{\xxl{#1}{#2}\xxr{#1}{#2}}
\newcommand\xxh[6]{
\draw(#1/10,#2/10)+(0.5*#3*45+0.5*#4*45:#6) node[above] {$#5$};}
\newcommand\xxhu[4][0.15]{\xxh{#2}{#3}13{#4}{#1}}
\newcommand\stree[1]{\XX{\xxhu[0.25]00{#1}}}

\nc{\dnx}{\Delta_n A} \nc{\dx}{\Delta A} \nc{\dgp}{{\rm deg_{P}}}
\nc{\dgt}{{\rm deg_{T}}} \nc{\dg}{{\rm deg}} \nc{\ida}{ID($A$)} \nc{\tu}{\tilde{u}} \nc{\tv}{\tilde{v}}
\nc{\nr}{\calr_n} \nc{\nz}{\calz_n} \nc{\fun}{\cala_{n,d}}
 \nc{\fbase}{\calb} \nc{\LF}{\mathrm{RF}} \nc{\FFA}{\mathrm{LF}} \nc{\irr}{\mathrm{Irr}}
 \nc{\result}{\bfk\mathrm{Irr}(S_n)}  \nc{\I}{I_{\mathrm{ID},n}^0}
 \nc{\nrs}{\calr_n^\star} \nc{\ii}{\mathrm{I}} \nc{\iii}{\mathrm{II}}
\nc{\intl}{{\rm int}}\nc{\ws}[1]{{#1}}\nc{\deleted}[1]{\delete{#1}}\nc{\plas}{placements\xspace}

\nc{\bim}[1]{#1}  \nc{\shaop}{\sha_{\Omega}^{+}}  \nc{\shao}{\sha_{\Omega}}
\nc{\bbim}[2]{#1 #2} \nc{\bbbim}[2]{#1,\, #2} \nc{\RBF}{{\rm RBF}}
\nc{\frb}{F_{\RB}} \nc{\shaf}{\ssha_{\tiny{\Omega}}} \nc{\sham}{\diamond_{\tiny{\Omega}}}
\nc{\lf}{\lfloor} \nc{\rf}{\rfloor} \nc{\shan}{\ssha_{\lambda}}
\nc{\rlex}{{\rm {lex}}} \nc{\bb}{\Box} \nc{\ra}{\rightarrow}
\nc{\e}{{\rm {e}}}
\nc{\DDF}{\mathrm{DD}(X,\,\Omega)}\nc{\DTF}{\mathrm{DT}(X,\,\Omega)} \nc{\DT}{\mathrm{DT}'(\Omega,\,V)}
\nc{\bra}{\mathrm{bra}} \nc{\bre}{\mathrm{bre}}
\nc{\dec}{\mathrm{dec}} \nc{\diamondw}{\diamond_{w}}
\nc{\type}{\mathrm{type}}

\nc\caF[1]{\cal{F}_{#1}(X,\,\Omega)}
\nc\calt{\cal{T}(X,\,\Omega)} \nc\caltn{\cal{T}_n(X,\,\Omega)}
\nc\caltbin{\cal{T}_b(X,\,\Omega)}
\nc\calta{\cal{T}_0(X,\,\Omega)}
\nc\caltb{\cal{T}_1(X,\,\Omega)}
\nc\caltc{\cal{T}_2(X,\,\Omega)}
\nc\caltd{\cal{T}_3(X,\,\Omega)}
\nc\caltm{\cal{T}_m(X,\,\Omega)}
\nc\calf{\cal{F}(X,\,\Omega)}
\nc\fram{\frak{M}(\Omega,\, X)}
\nc\shaw{\sha^{NC}_w(\Omega,\, X)}
\nc\dw{\diamond_w} \nc\dl{\diamond_\ell}
\nc\shal{\sha^{NC}_\ell(X,\, \Omega)} \nc\shav{\sha^{NC}_w(\Omega,\, V)} \nc\shat{\sha^{NC,1}_w(\Omega,\, T^{+}(V))}
\nc{\cfo}{\cal{F}(X,\,\Omega)}
\nc{\frat}{\mathfrak{T}}
\nc{\shh}{\mathrm{Sh}^+_{\Omega}}
\nc{\sh}{\mathrm{Sh}_{\Omega}}

\nc{\lar}{\varinjlim}
\nc\XO{(X,\,\Omega)}
\def\cxo#1#2;{\cal{#1}#2\XO}
\def\cxob#1#2;{\cal{#1}#2_b\XO}
\nc\lrf[2]{B_{#2}^+(#1)}
\nc{\fd}{\mathrm{\text{typed angularly decorated planar rooted trees}}}
\nc{\rb}{\mathrm{RBFWs}} \nc{\dfw}{\mathrm{DFW{(X)}}} \nc{\tfw}{\mathrm{TFW{(X)}}}
\nc{\tfv}{\mathrm{TFW{(V)}}} \nc{\rbf}{\mathrm{RBF}}

\def\Ve#1,#2,#3;{\vee_{#1,\,(#2,\,#3)}}
\def\bigv#1;#2;#3;{\bigvee\nolimits_{#1}^{#2;\,#3}}

\renewcommand{\geq}{\geqslant}
\renewcommand{\leq}{\leqslant}

\begin{document}

\title[Reduced typed  trees and generalized tridendriform algebras]{Reduced typed angularly decorated planar rooted trees and generalized tridendriform algebras}

\author{Lo\"\i c Foissy}
\address{Univ. Littoral C\^ote d'Opale, UR 2597 LMPA, Laboratoire de Math\'ematiques Pures et Appliqu\'ees Joseph Liouville F-62100 Calais, France}
\email{loic.foissy@univ-littoral.fr}

\author{Xiao-Song Peng}
\address{School of Mathematics and Statistics, Lanzhou University, Lanzhou, Gansu 730000, China}
\email{pengxiaosong3@163.com}

\date{\today}

\begin{abstract}
We introduce a generalization of tridendriform algebras, where each of the three products are replaced by a family
of products indexed by a set $\Omega$. We study the needed structure on $\Omega$ for free $\Omega$-tridendriform
algebras to be built on Schr\"oder trees (as it is the case in the classical case), with convenient decorations on their leaves.
We obtain in this way extended triassociative semigroups. We describe commutative $\Omega$-tridendriform algebras
in terms of typed words. We also study links with generalizations of Rota-Baxter algebras
and describe the Koszul duals of the corresponding operads.
\end{abstract}

\subjclass[2010]{17B38, 05C05, 16S10, 08B20
}

\keywords{Generalized tridendriform algebras; Semigroups; Schr\"oder trees; Typed words}

\maketitle

\tableofcontents

\setcounter{section}{0}

\allowdisplaybreaks
\section{Introduction}

Shuffle algebras have been studied for a long time, starting from combinatorial problems of card shufflings,
to operadic aspects. They were formalized in the fifties by Eilenberger and MacLane \cite{EMcL} and, independently,
by Sch\"utzenberger \cite{Sch}: a shuffle algebra is a vector space with a bilinear product $\prec$ following the axiom
\[(x\prec y)\prec y=x\prec (x\prec y+y\prec x).\]
Free shuffle algebras are based on words, with the half-shuffle product: for example,
if $a,b,c,d$ are letters,
\begin{align*}
a\prec bcd&=abcd,\\
ab\prec cd&=abcd+acbd+acdb,\\
abc\prec d&=abcd+abdc+adbc.
\end{align*}
It follows that the product $*$ defined by $x*y=x\prec  y+y\prec x$ is commutative and associative.
The noncommutative version, known as noncommutative shuffle algebras or dendriform algebras,
is introduced by Loday and Ronco \cite{LR}: a dendriform algebra is a vector space with two products $\prec$ and $\succ$,
with the following axioms:
\begin{align*}
(x\prec y)\prec z&=x\prec (y\prec z+y\succ z),\\
(x\succ y)\prec z&=x\succ (y\prec z),\\
x\succ (y\succ z)&=(x\prec y+x\succ y)\succ z.
\end{align*}
It follows that the product $*$ defined by $x*y=x\prec y+x\succ y$ is associative.
Loday and Ronco described the free dendriform algebra on one generator in terms of planar binary trees
\cite{LR2,Aguiar2}.
Recently, various generalizations of dendriform algebras appeared: the two products $\prec$ and  $\succ$
are replaced by families $(\prec_\alpha)_{\alpha \in \Omega}$ and $(\succ_\alpha)_{\alpha \in \Omega}$
of products parametrized by elements of a given set $\Omega$, which can have extra structures:
for matching dendriform algebras \cite{ZGG20}, $\Omega$ is just a set, whereras for family dendriform algebra it is a semigroup.
More generally, $\Omega$-dendriform algebras over an extended diassociative semigroup have been introduced and studied in
\cite{Foi20}. \\

In the spirit of \cite{Foi20}, we work here with parametrized versions of tridendriform algebras.
Tridendriform algebras are introduced in \cite{LRtridend}: they are vector spaces with three products $\prec$, $\succ$
and $\circ$, with the following axioms:
\begin{align*}
(a \prec b) \prec c&=\ a \prec (b \succ c)+ a \prec(b \prec c)+ a \prec (b \circ c) \\
(a \succ b) \prec c&=\ a \succ (b \prec c) \\
a \succ (b \succ c)&=\ (a \succ b) \succ c+(a \prec b) \succ c+ (a \circ b) \succ c \\
(a \succ b) \circ c&=\ a \succ (b \circ c) \\
(a \prec b) \circ c&=\ a\circ (b \succ c) \\
(a \circ b) \prec c&=\ a \circ (b \prec c)\\
(a \circ b) \circ c&=\ a \circ (b \circ c).
\end{align*}
Summing, we obtain that the product $*=\prec+\succ+\circ$ is associative; moreover, $\prec+\circ$ and $\succ$
define a dendriform algebra structure, as well as $\prec$ and $\succ+\circ$. A classical example of tridendriform algebra is given
by quasi-shuffle (or stuffle) algebras \cite{Hoffman,FoissyPatras}: if $(V,\cdot)$ is an associative algebra,
then the half-shuffle algebra $T(V)$ is tridendriform. For example, if $a,b,c,d\in V$,
\begin{align*}
ab \prec cd&=abcd+acbd+acdb+a(b\cdot c)d+ac(b\cdot d),\\
ab\succ cd&=cabd+cadb+cdab+c(a\cdot d)b+ca(b\cdot d),\\
ab\succ cd&=(a\cdot c)bd+(a\cdot c)db+(a\cdot c)(b\cdot d).
\end{align*}
Other examples of tridendriform algebras are based on packed words or on parking functions \cite{NovelliThibon}.
The free tridendriform algebra on one generator is described by Loday and Ronco in terms of planar reduced trees,
which we call here Schr\"oder trees.\\

We here give a parametrized version of tridendriform algebras. We start with a set $\Omega$
with six products $\leftarrow$, $\rightarrow$, $\lhd$, $\rhd$, $\ast$ and $\cdot$. An $\Omega$-tridendriform algebra
is given three families $(\prec_\alpha)_{\alpha \in \Omega}$,
$(\succ_\alpha)_{\alpha \in \Omega}$ and $(\circ_\alpha)_{\alpha \in \Omega}$ of bilinear products, with the following axioms:
\begin{align*}
(a \prec_{\alpha} b) \prec_{\beta} c&=\ a \prec_{\alpha \rightarrow \beta} (b \succ_{\alpha \rhd \beta} c)+ a \prec_{\alpha \leftarrow \beta} (b \prec_{\alpha \lhd \beta} c)+ a \prec_{\alpha \cdot \beta} (b \circ_{\alpha\ast \beta} c) \\
(a \succ_{\alpha} b) \prec_{\beta} c&=\ a \succ_{\alpha} (b \prec_{\beta} c) \\
a \succ_{\alpha} (b \succ_{\beta} c)&=\ (a \succ_{\alpha \rhd \beta} b) \succ_{\alpha \rightarrow \beta} c+(a \prec_{\alpha \lhd \beta} b) \succ_{\alpha \leftarrow \beta} c+ (a \circ_{\alpha\ast \beta} b) \succ_{\alpha \cdot \beta} c \\
(a \succ_{\alpha} b) \circ_{\beta} c&=\ a \succ_{\alpha} (b \circ_{\beta} c) \\
(a \prec_{\alpha} b) \circ_{\beta} c&=\ a\circ_{\beta} (b \succ_{\alpha} c) \\
(a \circ_{\alpha} b) \prec_{\beta} c&=\ a \circ_{\alpha} (b \prec_{\beta} c)\\
(a \circ_{\alpha} b) \circ_{\beta} c&=\ a \circ_{\alpha} (b \circ_{\beta} c).
\end{align*}
As in the classical case, particular examples of $\Omega$-tridendriform algebras
are given by $\Omega$-Rota-Baxter algebras, as defined in \cite{FoissyPeng}, see Proposition \ref{propRB}.
As a condition, we impose that free $\Omega$-tridendriform algebras are based on Schr\"oder trees
decorated in some sense by elements of $\Omega$.
In the dendriform case, the elements of $\Omega$ become types (that is to say decorations) of the internal edges
of the planar binary trees which were the basis of free dendriform algebras: this is not possible for Schr\"oder trees,
as the number of internal edges does not uniquely depend on the number of internal vertices,
which is a major difference with planar binary trees. Instead, we choose that the elements of $\Omega$
become decorations of the leaves of the Schr\"oder trees, at the exception of the leftmost and rightmost ones.
We then define inductively three families of products $\prec$, $\succ$ and $\circ$ on these trees
and ask for them to define an $\Omega$-tridendriform algebra. This impose a strong constraint on the products
taken on $\Omega$: they have to make $\Omega$ an extended triassociative semigroup (briefly, ETS):
see Definition \ref{defiSEDS} below for the list of 18 axioms defining these object below.
If this holds, these Schr\"oder trees indeed give free $\Omega$-tridendriform algebras (Theorem \ref{theoprincipal}).

Moreover, if $\Omega$ is a ETS and if $A$ is an $\Omega$-tridendriform algebra, then
$\bfk\Omega\otimes A$ inherits a structure of (classical) tridendriform algebra (Proposition \ref{prop3.6}),
which generalizes a similar result for $\Omega$-dendriform algebra proved in \cite{Foi20}.
We consider the particular case where $A$ is the $\Omega$-tridendriform algebra of Schr\"oder trees
in Proposition \ref{prop3.7}, where we give necessary and sufficient condition for $\bfk \Omega \ot \bfk \frat(X,\Omega)$
to be free.

We also study $\Omega$-tridendriform algebras of typed words. If $A$ is a matching associative algebra
(Definition \ref{defmatch}), we prove in Theorem \ref{theocomtridend} that the algebra $\sh^+(A)$
of $\Omega$-typed words on $A$ is an  $\Omega$-tridendriform algebra.
When $A$ and $\Omega$ are commutative, then we prove in Theorem \ref{theouniv} that
$\sh^+(A)$ is the free commutative $\Omega$-tridendriform over $A$. Consequently, if $A$
is the free matching algebra, generated by a set $X$,  then $\sh^+(A)$ is the free commutative $\Omega$-tridendriform
algebra generated by $X$.

The last section of the paper is devoted to the operad of $\Omega$-tridendriform algebras. In particular,
we study operadic morphisms from tridendriform algebras to $\Omega$-tridendriform algebras and
we compute its Koszul dual in Proposition \ref{propKoszul}, finding in this way a parametrised version
of triassociative algebras \cite{LRtridend}. \\

{\bf Notation.} Throughout this paper, let {\bfk} be a unitary commutative ring which will be the base
ring of all modules, algebras, as well as linear maps, unless otherwise specified.

\section{Generalized tridendriform algebras}

\subsection{EDS and ETS}
First, we recall the definition of diassociative semigroups and extended diassociative semigroups in~\cite{Foi20}.

\begin{defn}\cite[Definition~1]{Foi20}
A {\bf diassociative semigroup} is a family $(\Omega, \leftarrow, \rightarrow)$, where $\Omega$ is a set and $\leftarrow, \rightarrow: \Omega \times \Omega \rightarrow \Omega$ are maps such that
\begin{align*}
(\alpha \leftarrow \beta) \leftarrow \gamma&=\ \alpha \leftarrow (\beta \leftarrow \gamma)= \alpha \leftarrow (\beta \rightarrow \gamma),\\
(\alpha \rightarrow \beta) \leftarrow \gamma&=\ \alpha \rightarrow (\beta \leftarrow \gamma),\\
(\alpha \rightarrow \beta) \rightarrow \gamma&=\ (\alpha \leftarrow \beta) \rightarrow \gamma=\alpha \rightarrow (\beta \rightarrow \gamma),
\end{align*}
for all $\alpha, \beta, \gamma \in \Omega.$
\end{defn}

\begin{defn}\cite[Definition~2]{Foi20}
An {\bf extended diassociative semigroup} (abbr. EDS) is a family $(\Omega, \leftarrow, \rightarrow, \lhd,\rhd)$, where $\Omega$ is a set and $\leftarrow,\rightarrow, \lhd,\rhd: \Omega \times \Omega \rightarrow \Omega$ such that $(\Omega, \leftarrow, \rightarrow)$ is a diassociative semigroup and
\begin{align}
\label{EQ1}\alpha \rhd (\beta \leftarrow \gamma)&=\ \alpha \rhd \beta,\\
(\alpha \rightarrow \beta) \lhd \gamma&=\ \beta \lhd \gamma,\\
(\alpha \lhd \beta) \leftarrow ((\alpha \leftarrow \beta) \lhd \gamma)&=\ \alpha \lhd (\beta \leftarrow \gamma),\\
(\alpha \lhd \beta) \lhd ((\alpha \leftarrow \beta) \lhd \gamma)&=\ \beta \lhd \gamma,\\
(\alpha \lhd \beta) \rightarrow ((\alpha \leftarrow \beta) \lhd \gamma)&=\ \alpha \lhd (\beta \rightarrow \gamma),\\
(\alpha \lhd \beta) \rhd ((\alpha \leftarrow \beta) \lhd \gamma)&=\ \beta \rhd \gamma,\\
(\alpha \rhd (\beta \rightarrow \gamma)) \leftarrow (\beta \rhd \gamma)&=\ (\alpha \leftarrow \beta) \rhd \gamma,\\
(\alpha \rhd (\beta \rightarrow \gamma)) \lhd (\beta \rhd \gamma)&=\ \alpha \lhd \beta,\\
(\alpha \rhd (\beta \rightarrow \gamma)) \rightarrow (\beta \rhd \gamma)&=\ (\alpha \rightarrow \beta) \rhd \gamma,\\
\label{EQ10} (\alpha \rhd (\beta \rightarrow \gamma)) \rhd (\beta \rhd \gamma)&=\ \alpha \rhd \beta,
\end{align}
for all $\alpha,\beta, \gamma \in \Omega$.
\end{defn}

\begin{defn}\label{defiSEDS}
An {\bf extended triasssociative semigroup}  (abbr. ETS) is a family $(\Omega, \leftarrow, \rightarrow, \lhd,\rhd, \cdot, \ast)$, where $(\Omega, \leftarrow, \rightarrow, \lhd, \rhd)$ is an EDS and
\begin{align}
\label{EQ17}(\alpha \rightarrow \beta) \ast \gamma&=\ \beta \ast \gamma,\\
{\label{eq17}}(\alpha \rightarrow \beta) \cdot \gamma&=\ \alpha \rightarrow (\beta \cdot \gamma),\\
{\label{eq18}}\alpha \rhd \beta&=\ \alpha \rhd (\beta \cdot \gamma),\\
\label{EQ20}(\alpha \lhd \beta) \ast ((\alpha \leftarrow \beta) \lhd \gamma)&=\ \beta \ast \gamma,\\
{\label{eq19}}(\alpha \lhd \beta) \cdot ((\alpha \leftarrow \beta) \lhd \gamma)&=\ \alpha \lhd (\beta \cdot \gamma),\\
{\label{eq20}}(\alpha \leftarrow \beta) \leftarrow \gamma&=\ \alpha \leftarrow (\beta \cdot \gamma),\\
\label{EQ23}(\alpha \rhd (\beta \rightarrow \gamma))\ast (\beta \rhd \gamma)&=\ \alpha \ast \beta,\\
{\label{eq21}}\alpha \rightarrow (\beta \rightarrow \gamma)&=\ (\alpha \cdot \beta) \rightarrow \gamma,\\
{\label{eq22}}(\alpha \rhd (\beta \rightarrow \gamma)) \cdot (\beta \rhd \gamma)&=\ (\alpha \cdot \beta) \rhd \gamma,\\
\label{EQ26}(\alpha \leftarrow \beta)\ast \gamma&=\ \alpha\ast (\beta \rightarrow \gamma),\\
{\label{eq23}}(\alpha \leftarrow \beta) \cdot \gamma&=\ \alpha \cdot (\beta \rightarrow \gamma),\\
{\label{eq24}}\alpha \lhd \beta&=\ \beta \rhd \gamma,\\
\label{EQ29}\alpha\ast \beta&=\ \alpha\ast (\beta \leftarrow \gamma),\\
{\label{eq25}}(\alpha \cdot \beta) \lhd \gamma&=\ \beta \lhd \gamma,\\
{\label{eq26}}(\alpha \cdot \beta) \leftarrow \gamma&=\ \alpha \cdot (\beta \leftarrow \gamma),\\
\label{EQ32}\alpha \ast \beta&=\ \alpha \ast (\beta \cdot \gamma),\\
\label{EQ33}(\alpha \cdot \beta) \ast \gamma&=\ \beta \ast \gamma,\\
{\label{eq27}}(\alpha \cdot \beta) \cdot \gamma&=\ \alpha \cdot (\beta \cdot \gamma).
\end{align}
\end{defn}

\begin{exam}\label{ex2.4}
\begin{enumerate}
\item Let $(\Omega,\ast,\cdot)$ be a set with two products satisfying (\ref{EQ32})--(\ref{eq27}).
This holds for example if for any $\alpha,\beta \in \Omega$,
\begin{align*}
\alpha*\beta&=\alpha,&\alpha\cdot \beta&=\beta,
\end{align*}
or if for any $\alpha,\beta \in \Omega$,
\begin{align*}
\alpha*\beta&=\beta,&\alpha\cdot \beta&=\alpha.
\end{align*}
We put, for any $\alpha,\beta \in \Omega$:
\begin{align*}
\alpha \leftarrow \beta&=\alpha,&\alpha \lhd \beta&=\beta,\\
\alpha \rightarrow \beta&=\alpha,&\alpha \rhd \beta&=\alpha.
\end{align*}
Then $(\Omega, \leftarrow, \rightarrow, \lhd,\rhd, \cdot, \ast)$ is an ETS.
\item Let $(\Omega, \leftarrow, \rightarrow, \lhd,\rhd, \cdot, \ast)$ be an ETS.
For any $\alpha,\beta \in \Omega$, we put
\begin{align*}
\alpha \leftarrow^{op}\beta&=\beta \rightarrow \alpha,&
\alpha \lhd^{op}\alpha&=\beta \rhd \alpha,\\
\alpha \rightarrow^{op}\beta&=\beta \leftarrow \alpha,&
\alpha \rhd^{op}\alpha&=\beta \lhd \alpha,\\
\alpha \ast^{op}\beta&=\beta\ast\alpha,&
\alpha \cdot^{op}\beta&=\beta \cdot \alpha.
\end{align*}
Then $(\Omega,\leftarrow^{op},\rightarrow^{op},\lhd^{op},\rhd^{op},\ast^{op},\cdot^{op})$
 is also an ETS, called the opposite of $\Omega$. We shall say that $\Omega$ is commutative if it is equal to its opposite.
\end{enumerate}
\end{exam}

\subsection{$\Omega$-tridendriform algebras}

Let us now give the concept of $\Omega$-tridendriform algebras as follows.

\begin{defn}
Let $\Omega$ be a set with six products $\leftarrow, \rightarrow, \lhd, \rhd, \cdot, \ast$. An $\Omega$-{\bf tridendriform algebra} is a family $(A, (\prec_{\omega})_{\omega\in \Omega}, (\succ_{\omega})_{\omega \in \Omega}, (\circ_{\omega})_{\omega \in \Omega})$, where $A$ is a \bfk-module and $\prec_{\omega}, \succ_{\omega}, \circ_{\omega}: A \ot A \rightarrow A$ are linear maps such that
\begin{align}
(a \prec_{\alpha} b) \prec_{\beta} c&=\ a \prec_{\alpha \rightarrow \beta} (b \succ_{\alpha \rhd \beta} c)+ a \prec_{\alpha \leftarrow \beta} (b \prec_{\alpha \lhd \beta} c)+ a \prec_{\alpha \cdot \beta} (b \circ_{\alpha\ast \beta} c) \label{eq:tri1}\\
(a \succ_{\alpha} b) \prec_{\beta} c&=\ a \succ_{\alpha} (b \prec_{\beta} c) \label{eq:tri2}\\
a \succ_{\alpha} (b \succ_{\beta} c)&=\ (a \succ_{\alpha \rhd \beta} b) \succ_{\alpha \rightarrow \beta} c+(a \prec_{\alpha \lhd \beta} b) \succ_{\alpha \leftarrow \beta} c+ (a \circ_{\alpha\ast \beta} b) \succ_{\alpha \cdot \beta} c \label{eq:tri3}\\
(a \succ_{\alpha} b) \circ_{\beta} c&=\ a \succ_{\alpha} (b \circ_{\beta} c) \label{eq:tri4}\\
(a \prec_{\alpha} b) \circ_{\beta} c&=\ a\circ_{\beta} (b \succ_{\alpha} c) \label{eq:tri5}\\
(a \circ_{\alpha} b) \prec_{\beta} c&=\ a \circ_{\alpha} (b \prec_{\beta} c)\label{eq:tri6}\\
(a \circ_{\alpha} b) \circ_{\beta} c&=\ a \circ_{\alpha} (b \circ_{\beta} c) \label{eq:tri7}
\end{align}
for all $a,b,c \in A$ and $\alpha, \beta \in \Omega$. If moreover, for all $a,b\in A$ and $\alpha\in \Omega$,
\begin{align*}
a \prec_{\alpha} b&=b \succ_{\alpha} a && \text{ and }& a \circ_{\alpha}b&=b \circ_{\alpha} a,
\end{align*}
then $A$ is called a {\bf commutative $\Omega$-tridendriform algebra}.
\end{defn}

\begin{exam}
\begin{enumerate}
\item If $\Omega$ is a semigroup, we take all  maps $\circ_{\omega}$ with $\omega \in \Omega$ to be equal to
an associative product $\star$  and
\begin{align*}
\alpha \rightarrow \beta&=\alpha \leftarrow \beta=\alpha \cdot \beta=\alpha\star \beta,&
\alpha \rhd \beta&=\alpha,\,&\alpha \lhd \beta&=\beta,
\end{align*}
then $\Omega$-tridendriform algebra are tridendriform family algebras~\cite{ZG19}.

\item For a set $\Omega$, define, as in Example \ref{ex2.4},
\begin{align*}
\alpha \rightarrow \beta&=\alpha \lhd \beta=\alpha\ast \beta=\beta,&
\alpha \rhd \beta&=\alpha \leftarrow \beta=\alpha \cdot \beta=\alpha.
\end{align*}
Then $\Omega$-tridendriform algebra are matching tridendriform algebras~\cite{ZGG20}.
\end{enumerate}
\end{exam}

\begin{remark}
$\Omega$-dendriform algebras, as defined in \cite{Foi20}, are $\Omega$-tridendriform algebras such that,
for any $\alpha\in \Omega$, $\circ_\alpha=0$.
\end{remark}

Next we show the relationship between $\Omega$-Rota-Baxter algebras \cite{FoissyPeng}
 and $\Omega$-tridendriform algebras.
\begin{prop} \label{propRB}
Let $\Omega$ be a set with six products $\leftarrow, \rightarrow, \lhd, \rhd, \cdot, \ast$ and, for any $\alpha\in \Omega$,
let $\mu_\alpha \in \Omega$. We put, for any $\alpha,\beta \in \Omega$, $\lambda_{\alpha,\beta}=\mu_{\alpha\ast \beta}$.
Let $(A, (P_{\omega})_{\omega \in \Omega})$ be an $\Omega$-Rota-Baxter algebra of weight $\lambda$.
Then $(A, (\prec_{\omega})_{\omega \in \Omega}, (\succ_{\omega})_{\omega \in \Omega},
(\circ_{\omega})_{\omega \in \Omega})$ is an $\Omega$-tridendriform algebra, where
\begin{align*}
a \prec_{\omega} b&:=aP_{\omega}(b), & a \succ_{\omega} b&:=P_{\omega}(a)b, &a \circ_{\omega} b&:=\mu_{\omega} ab
\end{align*}
for all $a, b \in A$ and $\omega \in \Omega$.
\end{prop}

\begin{proof}
For $a,b,c \in A$ and $\alpha, \beta \in \Omega$,
\begin{align*}
(a \prec_{\alpha} b) \prec_{\beta} c&=\ (a P_{\alpha}(b) )P_{\beta}(c)=a(P_{\alpha}(b)P_{\beta}(c))\\
&=\ aP_{\alpha \rightarrow \beta}(P_{\alpha \rhd \beta}(b)c )+a P_{\alpha \leftarrow \beta} (b P_{\alpha \lhd \beta}(c) )+a P_{\alpha \cdot \beta} (\lambda_{\alpha, \beta} bc )\\
&=\ a \prec_{\alpha \rightarrow \beta} (b \succ_{\alpha \rhd \beta} c)+ a\prec_{\alpha \leftarrow \beta}(b \prec_{\alpha \lhd \beta} c)+ a \prec_{\alpha \cdot \beta} (b \circ_{\alpha\ast \beta} c)\\
(a \succ_{\alpha} b) \prec_{\beta} c&=\ (P_{\alpha}(a) b)P_{\beta}(c)=P_{\alpha}(a)(b P_{\beta}(c))=a \succ_{\alpha}(b \prec_{\beta} c)\\
a \succ_{\alpha}(b \succ_{\beta} c)&=\ P_{\alpha}(a) (P_{\beta}(b)c )=(P_{\alpha}(a)P_{\beta}(b) )c\\
&=\ P_{\alpha \rightarrow \beta}(P_{\alpha \rhd \beta}(a)b )c+P_{\alpha \leftarrow \beta}( a P_{\alpha \lhd\beta}(b))c+P_{\alpha \cdot \beta}(\lambda_{\alpha, \beta} ab)c\\
&=\ (a \succ_{\alpha \rhd \beta} b) \succ_{\alpha \rightarrow \beta} c+ (a \prec_{\alpha \lhd \beta} b) \succ_{\alpha \leftarrow \beta}c+(a \circ_{\alpha\ast \beta} b)\succ_{\alpha \cdot \beta} c\\
(a \succ_{\alpha} b) \circ_{\beta}c&=\ \mu_{\beta} (P_{\alpha}(a)b )c=P_{\alpha}(a) (\mu_{\beta} bc )=a \succ_{\alpha}(b \circ_{\beta} c)\\
(a \prec_{\alpha} b) \circ_{\beta} c&=\ \mu_{\beta}(a P_{\alpha}(b) )c=\mu_{\beta}a (P_{\alpha}(b)c )=a \circ_{\beta} (b \succ_{\alpha}c)\\
(a \circ_{\alpha} b) \prec_{\beta}c&=\mu_{\alpha} (ab) P_{\beta}(c)=\mu_{\alpha} a(bP_{\beta}(c))= a\circ_{\alpha}(b \prec_{\beta} c)\\
(a \circ_{\alpha} b) \circ_{\beta}c&=\ \mu_{\beta} \mu_{\alpha} (a b) c=\mu_{\alpha} \mu_{\beta}a(bc)=a \circ_{\alpha} (b \circ_{\beta} c). \qedhere
\end{align*}
\end{proof}

\section{Free $\Omega$-tridendriform algebras}

\subsection{$\Omega$-tridendriform algebras on leaf-typed angularly decorated Schr\"oder trees}
Recall from~\cite{ZGM} that  Schr\"oder trees are planar rooted trees such that there are at least two incoming edges for each vertex. For a Schr\"oder tree $T$, we still view the root and the leaves of $T$ as edges rather than vertices. Denote by $L(T)$ the set of leaf edges of $T$, i.e. edges which represent the leaves of $T$ and denote by $IL(T)$ the subset of $L(T)$ consisting of leaf edges which are neither the leftmost one nor the rightmost one.

\begin{defn}
Let $X$ and $\Omega$ be two sets. An $X$-{\bf angularly decorated $\Omega$-leaf typed} (abbr. {\bf leaf-typed angularly decorated}) {\bf Schr\"oder tree} is a triple $T = (T, \dec, \type)$, where $T$ is a Schr\"oder tree,  $\dec : A(T) \rightarrow X$ and $\type : IL(T) \rightarrow  \Omega$ are maps.
\end{defn}

For $n \geq 1$, let $\frat_n(X,\Omega)$ be the set of $X$-angularly decorated $\Omega$-leaf typed Schr\"oder trees with $n+1$ leaves and at lease one internal vertex. Denote by
\begin{align*}
\frat(X,\Omega)&:=\bigsqcup_{n\geq 1}\frat_n(X,\Omega)&&\text{ and }&\bfk\frat(X,\Omega) &
:=\bigoplus_{n\geq 1}\bfk \frat_n(X,\Omega).
\end{align*}

Here are some examples.
\begin{align*}
\frat_1(X,\Omega)&= \ \left\{\stree x\Bigm|x\in X\right\},\\
 \frat_2(X,\Omega)&=\left\{
\XX{\xxr{-5}5
\node at (0.2,1.2) {\tiny $\alpha$};
\xxhu00x \xxhu{-5}5y
},
\XX{\xxl55
\node at (-0.2,1.2) {\tiny $\alpha$};
\xxhu00x \xxhu55y
},
\XX{\xx002
\xxh0023{x\ }{0.5} \xxh0012{\ \,y}{0.4}
\node at (0,1.2) {\tiny $\alpha$};
}\Bigm|x,y\in X,\alpha\in\Omega
\right\},\\
\frat_3(X,\Omega)&= \ \left\{
\XX[scale=1.6]{\xxr{-4}4\xxr{-7.5}{7.5}
\node at (0.3,1.2) {\tiny $\alpha$};
\node at (-0.3,1.2) {\tiny $\beta$};
\xxhu00x \xxhu[0.1]{-4}4{\,y} \xxhu[0.1]{-7.5}{7.5}{z}
},
\XX[scale=1.6]{\xxr{-5}5\xxl{-2}8
\node at (-0.5,1.2) {\tiny $\alpha$};
\node at (0.2,1.2) {\tiny $\beta$};
\xxhu00x
\xxhu[0.1]{-5}5{y} \xxhu[0.1]{-2}8{z}
},
\XX[scale=1.6]{\xxr{-4}4\xx{-4}42
\node at (-0.4,1.2) {\tiny $\alpha$};
\node at (0.3,1.2) {\tiny $\beta$};
\xxhu00x \xxh{-4}423{y\ \,}{0.3} \xxh{-4}412{\ \, z}{0.3}
},
\XX[scale=1.6]{\xx00{1.6}\xx00{2.4}
\xxh001{1.6}{\ \ \,z}{0.6}
\xxh00{1.6}{2.4}{y}{0.5}
\xxh00{2.4}3{x\ \ }{0.6}
\node at (0.3,1.2) {\tiny $\beta$};
\node at (-0.3,1.2) {\tiny $\alpha$};
},
\XX[scale=1.6]{\xxlr0{7.5} \draw(0,0)--(0,0.75);
\xxh0023{x\ \,}{0.3} \xxhu[0.12]0{7.5}{z}
\node at (0.16,0.35) {\tiny $y$};
\node at (0.3,1.2) {\tiny $\beta$};
\node at (-0.3,1.2) {\tiny $\alpha$};
},
\ldots\Bigg|\,x,y,z\in X,\alpha,\beta\in\Omega
\right\}.
\end{align*}

For $T \in \frat(X,\Omega)$ and $\omega \in \Omega$, let $^{\omega}T$ be the tree $T$ whose leftmost leaf edge is typed by $\omega$ and
let $T^{\omega}$ be the tree $T$ whose rightmost leaf edge is typed by $\omega$. We also define $^{\omega}|:=|^{\omega}:=|$ and say $l(^{\omega}T)=r(T^{\omega})=\omega$ if $T \neq |$. Graphically, an element $T \in \frat(X, \Omega)$ is of the form
\begin{align*}
T=\treeoo{\cdb o\ocdx[2]{o}{a1}{160}{T_1}{left}
\ocdx[2]{o}{a2}{120}{T_2}{above}
\ocdx[2]{o}{a3}{60}{T_m}{above}
\ocdx[2]{o}{a4}{20}{T_{m+1}}{right}
\node at (140:\xch) {$x_1$};
\node at (90:\xch) {$\cdots$};
\node at (40:\xch) {$x_m$};
},
\end{align*}
with $n \geq 1$, $x_1, \cdots, x_{n} \in X$ and $T_1=|$ or $U_1^{\alpha_{1,r}}$, $T_i=|$ or $\prescript{\alpha_{i,l}}{}{U_i}^{\alpha_{i,r}}$ for $2\leq i \leq m$ and $T_{m+1}=|$ or $\prescript{\alpha_{m+1,l}}{}{U_{m+1}}$ for some $U_j \in \frat(X,\Omega)$ and $\alpha_{j,l}, \alpha_{j,r} \in \Omega$. For each $T\in \frat(X,\Omega)$, denote by $\leaf(T)$ the number of leaf edges of $T$.\\

Now, let $\Omega$ be a set with six products $\leftarrow, \rightarrow, \lhd, \rhd, \cdot, \ast$. For $\omega \in \Omega$, we define products $\prec_{\omega}, \succ_{\omega}, \circ_{\omega}$ on $\bfk \frat(X,\Omega)$ recursively as follows.

For
\begin{align*}
T&=\treeoo{\cdb o\ocdx[2]{o}{a1}{160}{T_1}{left}
\ocdx[2]{o}{a2}{120}{T_2}{above}
\ocdx[2]{o}{a3}{60}{T_m}{above}
\ocdx[2]{o}{a4}{20}{T_{m+1}}{right}
\node at (140:\xch) {$x_1$};
\node at (90:\xch) {$\cdots$};
\node at (40:\xch) {$x_m$};
}&&\text{ and }& U&=\treeoo{\cdb o\ocdx[2]{o}{a1}{160}{U_1}{left}
\ocdx[2]{o}{a2}{120}{U_2}{above}
\ocdx[2]{o}{a3}{60}{U_n}{above}
\ocdx[2]{o}{a4}{20}{U_{n+1}}{right}
\node at (140:\xch) {$y_1$};
\node at (90:\xch) {$\cdots$};
\node at (40:\xch) {$y_n$};
}\, \in \frat(X,\Omega),
\end{align*}
we define $T \prec_{\omega} U, T \succ_{\omega} U, T \circ_{\omega} U$ by induction on $\leaf(T)+\leaf(U)$.
If $\leaf(T)+\leaf(U)=4$, we have
\begin{align*}
T&=\stree x && \text{ and }& U&=\stree y.
\end{align*}
 Then
\begin{align*}
T\prec_\omega U &:=\
\XX{\xxl55
\node at (-0.2,1.2) {\tiny $\omega$};
\xxhu00x \xxhu55y
},&T\succ_\omega U&:= \XX{\xxr{-5}5
\node at (0.2,1.2) {\tiny $\omega$};
\xxhu00y \xxhu{-5}5x
}&& \text{ and }&
T\circ_{\omega} U&:= \XX{\xx002
\xxh0023{x\ }{0.5} \xxh0012{\ \,y}{0.4}
\node at (0,1.3) {\tiny $\omega$};
}.
\end{align*}
For the induction step of $\leaf(T)+\leaf(U)\geq 5$, to define $T\prec_\omega U$, we consider the following two cases.

\noindent {\bf Case 1:} $T_{m+1}=|$. Then
\begin{align*}
T\prec_\omega U:&=\ \treeoo{\cdb o\ocdx[2]{o}{a1}{160}{T_{1}}{left}
\ocdx[2]{o}{a2}{120}{T_{2}}{above}
\ocdx[2]{o}{a3}{60}{T_{m}}{above}
\ocdx[2]{o}{a4}{20}{\prescript{\omega}{}{U}}{right}
\node at (140:\xch) {$x_1$};
\node at (90:\xch) {$\cdots$};
\node at (40:\xch) {$x_m$};
}.
\end{align*}
\noindent {\bf Case 2:} $T_{m+1}\neq |$ and $l(T_{m+1})=\alpha_{m+1}$. Then
\begin{align*}
T\prec_\omega U:&=\ \treeoo{\cdb o\ocdx[2]{o}{a1}{160}{T_{1}}{left}
\ocdx[2]{o}{a2}{120}{T_{2}}{above}
\ocdx[2]{o}{a3}{60}{T_{m}}{above}
\cdx[2]{o}{a4}{20}
\node at (140:\xch) {$x_1$};
\node at (90:\xch) {$\cdots$};
\node at (40:\xch) {$x_m$};
} \left(\begin{array}{c}
\prescript{\alpha_{m+1} \rightarrow \omega}{}{T_{m+1}} \succ_{\alpha_{m+1} \rhd \omega} U\\
+\prescript{\alpha_{m+1} \leftarrow \omega}{}{T_{m+1}} \prec_{\alpha_{m+1} \lhd \omega} U\\
 +\prescript{\alpha_{m+1} \cdot \omega}{}{T_{m+1}} \circ_{\alpha_{m+1} \ast \omega} U
 \end{array}\right).
\end{align*}

To define $T\succ_\omega U$, we consider the following two cases.

\noindent{\bf Case 3:} $U_{1}=|$. Then
\begin{align*}
T\succ_\omega U:&=\ \treeoo{\cdb o\ocdx[2]{o}{a1}{160}{T^{\omega}}{left}
\ocdx[2]{o}{a2}{120}{U_{2}}{above}
\ocdx[2]{o}{a3}{60}{U_{n}}{above}
\ocdx[2]{o}{a4}{20}{U_{n+1}}{right}
\node at (140:\xch) {$y_1$};
\node at (90:\xch) {$\cdots$};
\node at (40:\xch) {$y_n$};
}.
\end{align*}
\noindent {\bf Case 4:} $U_{1}\neq |$ and $r(U_1)=\beta_1$. Then
\begin{align*}
T\succ_\omega U:&=\left(\begin{array}{c}
T \succ_{\omega \rhd \beta_1} U_1^{\omega \rightarrow \beta_{1}}\\
+T \prec_{\omega \lhd \beta_1} U_1^{\omega \leftarrow \beta_1}\\
+T \circ_{\omega \ast \beta_1} U_1^{\omega \cdot \beta_1}
\end{array}\right)
\treeoo{\cdb o\cdx[2]{o}{a1}{160}
\ocdx[2]{o}{a2}{120}{U_{2}}{above}
\ocdx[2]{o}{a3}{60}{U_{n}}{above}
\ocdx[2]{o}{a4}{20}{U_{n+1}}{right}
\node at (140:\xch) {$y_1$};
\node at (90:\xch) {$\cdots$};
\node at (40:\xch) {$y_n$};
}.
\end{align*}
To define $T\circ_{\omega} U$, we consider the following four cases.

\noindent{\bf Case 5:} $T=\stree x$. Then
\begin{align*}
T\circ_{\omega} U:&=\
\treeoo{\cdb o
\cdx[2]{o}{a1}{160}
\ocdx[2]{o}{a2}{120}{\prescript{\omega}{}{U_{1}}}{above}
\ocdx[2]{o}{a3}{60}{U_n}{above}
\ocdx[2]{o}{a4}{20}{U_{n+1}}{right}
\node at (140:\xch) {$x$};
\node at (90:\xch) {$\cdots$};
\node at (40:\xch) {$y_n$};
}.
\end{align*}

\noindent{\bf Case 6:} $T=\treeoo{\cdb o
\cdx[2]{o}{a2}{120}
\ocdx[2]{o}{a3}{60}{\prescript{\alpha}{}{T_2}}{above}
\node at (90:\xch) {$x$};
}$. Then
\begin{align*}
T\circ_{\omega} U:&=\ \stree{x} \circ_{\omega} (T_2 \succ_{\alpha} U).
\end{align*}

\noindent {\bf Case 7:} $T= \treeoo{\cdb o\cdx[2]{o}{a1}{160}
\ocdx[2]{o}{a2}{120}{\prescript{\alpha}{}{T_2}}{above}
\ocdx[2]{o}{a3}{60}{T_{m}}{above}
\ocdx[2]{o}{a4}{20}{T_{m+1}}{right}
\node at (140:\xch) {$x_1$};
\node at (90:\xch) {$\cdots$};
\node at (40:\xch) {$x_m$};
}$ with $m \geq 2$. Then
\begin{align*}
T\circ_{\omega} U:&=\ \stree {x_1} \circ_{\alpha} \left(\treeoo{\cdb o
\ocdx[2]{o}{a2}{135}{T_2}{above}
\ocdx[2]{o}{a3}{90}{T_{m}}{above}
\ocdx[2]{o}{a4}{45}{T_{m+1}}{right}
\node at (110:\xch) {$\cdots$};
\node at (70:\xch) {$x_m$};
} \circ_{\omega} U \right).
\end{align*}

\noindent {\bf Case 8:} $T=\treeoo{\cdb o\ocdx[2]{o}{a1}{160}{T_1^{\alpha}}{left}
\ocdx[2]{o}{a2}{120}{T_2}{above}
\ocdx[2]{o}{a3}{60}{T_m}{above}
\ocdx[2]{o}{a4}{20}{T_{m+1}}{right}
\node at (140:\xch) {$x_1$};
\node at (90:\xch) {$\cdots$};
\node at (40:\xch) {$x_m$};
}$ with $T_1 \neq |$. Then
{\small{\begin{align*}
T\circ_{\omega} U:&=\ T_1 \succ_{\alpha}\left(\treeoo{\cdb o\cdx[2]{o}{a1}{160}
\ocdx[2]{o}{a2}{120}{T_2}{above}
\ocdx[2]{o}{a3}{60}{T_m}{above}
\ocdx[2]{o}{a4}{20}{T_{m+1}}{right}
\node at (140:\xch) {$x_1$};
\node at (90:\xch) {$\cdots$};
\node at (40:\xch) {$x_m$};
} \circ_{\omega} U\right).
\end{align*}}}
Note that for $T, U \in \frat(X,\Omega)$ with $\leaf(T)=m,\leaf(U)=n$, then the trees in $T\prec_{\omega} U, T \succ_{\omega} U, T\circ_{\omega} U$ which are defined in Case~1-7 have  $m+n-1$ leaf edges. Hence Case~8 is in the induction step.

Let $j: X \rightarrow \bfk \frat(X, \Omega), x \mapsto \stree x$ be the natural inclusion. Then with these products defined as above, we obtain the following result:
\begin{theorem}\label{theoprincipal}
Let $\Omega$ be a set with six products $\leftarrow, \rightarrow, \lhd, \rhd, \cdot, \ast$. Then the following conditions are equivalent:
\begin{enumerate}
\item \label{it11} With these products, $\bfk \frat(X,\Omega)$  and the map $j$ is the free $\Omega$-tridendriform algebra generated by $X$.
\item \label{it22} With these products, $\bfk \frat(X,\Omega)$ is an $\Omega$-tridendriform algebra.
\item \label{it33} $(\Omega, \leftarrow, \rightarrow, \lhd, \rhd, \cdot,\ast)$ is an ETS.
\end{enumerate}
\end{theorem}

\begin{proof}
\ref{it11} $\Longrightarrow$ \ref{it22} It is obvious.

\ref{it22} $\Longrightarrow$ \ref{it33} For $\alpha, \beta, \gamma \in \Omega$ and $\stree a, \stree b, \XX{\xxr{-5}5
\node at (0.2,1.1) {$\gamma$};
\xxhu00d \xxhu{-5}5c
} \in \bfk \frat(X,\Omega)$, by calculation there are both eleven trees in the expression of
\[\stree a \succ_{\alpha} \left(\stree b \succ_{\beta} \XX{\xxr{-5}5
\node at (0.2,1.1) {$\gamma$};
\xxhu00d \xxhu{-5}5c
}\right )\]
and of
\begin{align*}
&\left(\stree a \succ_{\alpha \rhd \beta} \stree b\right) \succ_{\alpha \rightarrow \beta} \XX{\xxr{-5}5
\node at (0.2,1.1) {$\gamma$};
\xxhu00d \xxhu{-5}5c
} +\left(\stree a \prec_{\alpha \lhd \beta} \stree b\right) \succ_{\alpha \leftarrow \beta} \XX{\xxr{-5}5
\node at (0.2,1.1) {$\gamma$};
\xxhu00d \xxhu{-5}5c
}\\
&+ \left(\stree a \circ_{\alpha\ast \beta} \stree b\right) \succ_{\alpha \cdot \beta} \XX{\xxr{-5}5
\node at (0.2,1.1) {$\gamma$};
\xxhu00d \xxhu{-5}5c
}.\end{align*}

 Identifying the types of the trees in these expressions, we get that $(\Omega, \leftarrow, \rightarrow, \lhd, \rhd, \cdot,\ast)$ is an ETS.

\ref{it33} $\Longrightarrow$ \ref{it22}
Extending the products $\prec_{\omega}, \succ_{\omega}, \circ_{\omega}$ to the space
\[\Big(\bfk \frat(X,\Omega) \ot \bfk \frat(X,\Omega) \Big)\oplus\Big(\bfk |\ot\bfk \frat(X,\Omega)\Big)\oplus \Big(\bfk \frat(X,\Omega) \ot\bfk |\Big)\]
by
\begin{align*}
|\succ_\omega T&:=T\prec_\omega |:=T,&
 |\prec_\omega T&:=T\succ_\omega |:=0 &&\text{ and }& |\circ_{\omega} T&:=T\circ_{\omega} |:=0,
\end{align*}
for $\omega\in\Omega$ and $T\in \bfk \frat(X,\Omega)$.
By convention, we consider the added element $\emptyset$ as a unit for the six products of $\Omega$. Then the products $\prec_{\omega}, \succ_{\omega}$ can be rewritten in the following way: for
\begin{align*}
T&=\treeoo{\cdb o\ocdx[2]{o}{a1}{160}{T_1}{left}
\ocdx[2]{o}{a2}{120}{T_2}{above}
\ocdx[2]{o}{a3}{60}{T_m}{above}
\ocdx[2]{o}{a4}{20}{\prescript{\alpha_{m+1}}{}{T_{m+1}}}{right}
\node at (140:\xch) {$x_1$};
\node at (90:\xch) {$\cdots$};
\node at (40:\xch) {$x_m$};
}&&\text{ and }& U&=\treeoo{\cdb o\ocdx[2]{o}{a1}{160}{U_1^{\beta_1}}{left}
\ocdx[2]{o}{a2}{120}{U_2}{above}
\ocdx[2]{o}{a3}{60}{U_n}{above}
\ocdx[2]{o}{a4}{20}{U_{n+1}}{right}
\node at (140:\xch) {$y_1$};
\node at (90:\xch) {$\cdots$};
\node at (40:\xch) {$y_n$};
}\, \in \frat(X,\Omega),
\end{align*}
then
\begin{align*}
T \prec_{\omega} |&=\ | \succ_{\omega} T=T,\\
| \prec_{\omega} T&=\ T \succ_{\omega} |=0,\\
|\circ_{\omega} T&=\ T\circ_{\omega} |=0,\\
T\prec_\omega U&:= \treeoo{\cdb o\ocdx[2]{o}{a1}{160}{T_{1}}{left}
\ocdx[2]{o}{a2}{120}{T_{2}}{above}
\ocdx[2]{o}{a3}{60}{T_{m}}{above}
\cdx[2]{o}{a4}{20}
\node at (140:\xch) {$x_1$};
\node at (90:\xch) {$\cdots$};
\node at (40:\xch) {$x_m$};
}\left(\begin{array}{c}
\prescript{\alpha_{m+1} \rightarrow \omega}{}{T_{m+1}} \succ_{\alpha_{m+1} \rhd \omega} U\\
+\prescript{\alpha_{m+1} \leftarrow \omega}{}{T_{m+1}} \prec_{\alpha_{m+1} \lhd \omega} U\\
 +\prescript{\alpha_{m+1} \cdot \omega}{}{T_{m+1}} \circ_{\alpha_{m+1} \ast \omega} U
\end{array}\right),\\
T\succ_\omega U&:=\left(\begin{array}{c}
T \succ_{\omega \rhd \beta_1} U_1^{\omega \rightarrow \beta_{1}}\\
+T \prec_{\omega \lhd \beta_1} U_1^{\omega \leftarrow \beta_1}\\
+T \circ_{\omega \ast \beta_1} U_1^{\omega \cdot \beta_1}
\end{array}\right)
\treeoo{\cdb o\cdx[2]{o}{a1}{160}
\ocdx[2]{o}{a2}{120}{U_{2}}{above}
\ocdx[2]{o}{a3}{60}{U_{n}}{above}
\ocdx[2]{o}{a4}{20}{U_{n+1}}{right}
\node at (140:\xch) {$y_1$};
\node at (90:\xch) {$\cdots$};
\node at (40:\xch) {$y_n$};
}.\\
%
\end{align*}

We first show that $\bfk \frat(X,\Omega)$ is an $\Omega$-tridendriform algebra. We prove Eqs.~(\ref{eq:tri1})-(\ref{eq:tri7}) hold for \begin{align*}
 T&=\treeoo{\cdb o\ocdx[2]{o}{a1}{160}{T_1}{left}
\ocdx[2]{o}{a2}{120}{T_2}{above}
\ocdx[2]{o}{a3}{60}{T_l}{above}
\ocdx[2]{o}{a4}{20}{T_{l+1}}{right}
\node at (140:\xch) {$x_1$};
\node at (90:\xch) {$\cdots$};
\node at (40:\xch) {$x_m$};
},&U&=\treeoo{\cdb o\ocdx[2]{o}{a1}{160}{U_1}{left}
\ocdx[2]{o}{a2}{120}{U_2}{above}
\ocdx[2]{o}{a3}{60}{U_m}{above}
\ocdx[2]{o}{a4}{20}{U_{m+1}}{right}
\node at (140:\xch) {$y_1$};
\node at (90:\xch) {$\cdots$};
\node at (40:\xch) {$y_n$};
}&&\text{ and }&V&=\treeoo{\cdb o\ocdx[2]{o}{a1}{160}{V_1}{left}
\ocdx[2]{o}{a2}{120}{V_2}{above}
\ocdx[2]{o}{a3}{60}{V_n}{above}
\ocdx[2]{o}{a4}{20}{V_{n+1}}{right}
\node at (140:\xch) {$z_1$};
\node at (90:\xch) {$\cdots$};
\node at (40:\xch) {$z_n$};
}\, \in \frat(X,\Omega)
\end{align*}
by induction on the sum $p:=\leaf(T)+\leaf(U)+\leaf(V)$.
If $p=6$, then $\leaf(T)=\leaf(U)=\leaf(V)=2$ and $T, U, V$ are of the form
\begin{align*}
T &=\stree x,&U&=\stree y && \text{ and } & V&=\stree z.
\end{align*}
Eqs.~(\ref{eq:tri1})-(\ref{eq:tri7}) hold by direct calculation.

Suppose that Eqs.~(\ref{eq:tri1})-(\ref{eq:tri7}) hold for $p \leq q$, where $q \geq 6$ is a fixed positive integer. Consider the case of $p=q+1$. First, we prove Eq.~(\ref{eq:tri1}) and we assume $l(T_{l+1})=\alpha_{l+1}$ if $T_{l+1} \neq |$. Then
\begin{align*}
&\ (T \prec_{\alpha} U) \prec_{\beta} V\\
&=\  \left( \treeoo{\cdb o\ocdx[2]{o}{a1}{160}{T_{1}}{left}
\ocdx[2]{o}{a2}{120}{T_{2}}{above}
\ocdx[2]{o}{a3}{60}{T_{l}}{above}
\cdx[2]{o}{a4}{20}
\node at (140:\xch) {$x_1$};
\node at (90:\xch) {$\cdots$};
\node at (40:\xch) {$x_l$};
} \left(\begin{array}{c}
\prescript{\alpha_{m+1} \rightarrow \alpha}{}{T_{m+1}} \succ_{\alpha_{m+1} \rhd \alpha} U\\
+\prescript{\alpha_{m+1} \leftarrow \alpha}{}{T_{m+1}} \prec_{\alpha_{m+1} \lhd \alpha} U\\
\\ +\prescript{\alpha_{m+1} \cdot \alpha}{}{T_{m+1}} \circ_{\alpha_{m+1} \ast \alpha} U
\end{array}\right)\right) \prec_{\beta} V\\
&=\  \treeoo{\cdb o\ocdx[2]{o}{a1}{160}{T_{1}}{left}
\ocdx[2]{o}{a2}{120}{T_{2}}{above}
\ocdx[2]{o}{a3}{60}{T_{l}}{above}
\cdx[2]{o}{a4}{20}
\node at (140:\xch) {$x_1$};
\node at (90:\xch) {$\cdots$};
\node at (40:\xch) {$x_l$};
} \left(
\begin{array}{c}
(\prescript{(\alpha_{l+1} \rightarrow \alpha) \rightarrow \beta}{}{T_{l+1}} \succ_{\alpha_{l+1} \rhd \alpha} U) \succ_{(\alpha_{l+1}\rightarrow \alpha) \rhd \beta} V \\
\ +(\prescript{(\alpha_{l+1} \rightarrow \alpha)\leftarrow \beta}{}{T_{l+1}} \succ_{\alpha_{l+1} \rhd \alpha} U) \prec_{(\alpha_{l+1} \rightarrow \alpha) \lhd \beta} V \\
+(\prescript{(\alpha_{l+1} \rightarrow \alpha)\cdot \beta}{}{T_{l+1}} \succ_{\alpha_{l+1} \rhd \alpha} U) \circ_{(\alpha_{l+1} \rightarrow \alpha) \ast \beta} V \\
\ +(\prescript{(\alpha_{l+1} \leftarrow \alpha) \rightarrow \beta}{}{T_{l+1}} \prec_{\alpha_{l+1} \lhd \alpha} U)\succ_{(\alpha_{l+1} \leftarrow \alpha)\rhd \beta} V \\
+(\prescript{(\alpha_{l+1} \leftarrow \alpha) \leftarrow \beta}{}{T_{l+1}} \prec_{\alpha_{l+1} \lhd \alpha} U) \prec_{(\alpha_{l+1} \leftarrow \alpha) \lhd \beta} V \\
\ +(\prescript{(\alpha_{l+1} \leftarrow \alpha) \cdot \beta}{}{T_{l+1}} \prec_{\alpha_{l+1} \lhd \alpha} U) \circ_{(\alpha_{l+1} \leftarrow \alpha) \ast \beta} V\\
+(\prescript{(\alpha_{l+1} \cdot \alpha) \rightarrow \beta}{}{T_{l+1}} \circ_{\alpha_{l+1} \ast \alpha} U) \succ_{(\alpha_{l+1} \cdot \alpha) \rhd \beta} V \\
\ +(\prescript{(\alpha_{l+1} \cdot \alpha) \leftarrow \beta}{}{T_{l+1}} \circ_{\alpha_{l+1}\ast \alpha} U) \prec_{(\alpha_{l+1} \cdot \alpha) \lhd \beta} V\\
+(\prescript{(\alpha_{l+1} \cdot \alpha) \cdot \beta}{}{T_{l+1}} \circ_{\alpha_{l+1} \ast \alpha} U) \circ_{(\alpha_{l+1} \cdot \alpha) \ast \beta} V
\end{array}\right)\\
&=\  \treeoo{\cdb o\ocdx[2]{o}{a1}{160}{T_{1}}{left}
\ocdx[2]{o}{a2}{120}{T_{2}}{above}
\ocdx[2]{o}{a3}{60}{T_{l}}{above}
\cdx[2]{o}{a4}{20}
\node at (140:\xch) {$x_1$};
\node at (90:\xch) {$\cdots$};
\node at (40:\xch) {$x_l$};
}\left(\begin{array}{c}
\prescript{\alpha_{l+1} \rightarrow (\alpha \rightarrow \beta)}{}{T_{l+1}} \succ_{\alpha_{l+1}\rhd (\alpha \rightarrow \beta)} (U \succ_{\alpha \rhd \beta} V)\\
 +\prescript{\alpha_{l+1} \leftarrow (\alpha \rightarrow \beta)}{}{T_{l+1}} \prec_{\alpha_{l+1} \lhd (\alpha \rightarrow \beta)}(U \succ_{\alpha \rhd \beta} V)\\
 +\prescript{\alpha_{l+1} \cdot(\alpha \rightarrow \beta)}{}{T_{l+1}} \circ_{\alpha_{l+1} \ast (\alpha \rightarrow \beta)}(U \succ_{\alpha \rhd \beta} V)\\
 +\prescript{\alpha_{l+1} \rightarrow (\alpha \leftarrow \beta)}{}{T_{l+1}} \succ_{\alpha_{l+1} \rhd(\alpha \leftarrow \beta)} (U \prec_{\alpha \lhd \beta} V)\\
 +\prescript{\alpha_{l+1} \leftarrow (\alpha \leftarrow \beta)}{}{T_{l+1}} \prec_{\alpha_{l+1} \lhd (\alpha \leftarrow \beta)}(U \prec_{\alpha \lhd \beta} V) \\
 +\prescript{\alpha_{l+1} \cdot (\alpha \leftarrow \beta)}{}{T_{l+1}} \circ_{\alpha_{l+1} \ast (\alpha \leftarrow \beta)} (U \prec_{\alpha \lhd \beta} V) \\
 +\prescript{\alpha_{l+1} \rightarrow (\alpha \cdot \beta)}{}{T_{l+1}} \succ_{\alpha_{l+1} \rhd (\alpha \cdot \beta)} (U \circ_{\alpha \ast \beta} V) \\
 +\prescript{\alpha_{l+1} \leftarrow (\alpha \cdot \beta)}{}{T_{l+1}} \prec_{\alpha_{l+1} \lhd (\alpha \cdot \beta)} (U \circ_{\alpha \ast \beta} V) \\
 +\prescript{\alpha_{l+1} \cdot(\alpha \cdot \beta)}{}{T_{l+1}} \circ_{\alpha_{l+1} \ast (\alpha \cdot \beta)}(U \circ_{\alpha \ast \beta} V)
\end{array} \right)\\
&\ \hspace{4cm} \text{(by induction hypothesis and $(\Omega, \leftarrow, \rightarrow, \lhd, \rhd, \cdot, \ast)$ being an ETS)}\\
&=\ T \prec_{\alpha \rightarrow \beta} (U \succ_{\alpha \rhd \beta} V)+T \prec_{\alpha \leftarrow \beta} (U \prec_{\alpha \lhd \beta} V)+ T\prec_{\alpha \cdot \beta} (U \circ_{\alpha \ast \beta} V).
\end{align*}
Hence Eq.~(\ref{eq:tri1}) holds. Eq.(\ref{eq:tri3}) can be proved similarly.  Eq.~(\ref{eq:tri2}) holds directly as $T \succ_{\alpha} U$ changes the leftmost branch of $U$ and $U \prec_{\beta}V$ changes the rightmost branch of $U$ and $U$ has at least two branches. Now we show Eq.~(\ref{eq:tri4}) holds. If $U_1=|$, then $(T \succ_{\alpha} U) \circ_{\beta} V=T \succ_{\alpha} (U \circ_{\beta} V)$ by the definition of $\circ_{\omega}$ in Case 8. If $U_1 \neq |$, we assume $r(U_1)=\beta_1$, then
\begin{align*}
&\ (T \succ_{\alpha} U) \circ_{\beta} V\\
&=\ \left( \big(T \succ_{\alpha \rhd \beta_1} U_1^{\alpha \rightarrow \beta_{1}}+T \prec_{\alpha \lhd \beta_1} U_1^{\alpha \leftarrow \beta_1}+T \circ_{\alpha \ast \beta_1} U_1^{\alpha \cdot \beta_1}\big)
\treeoo{\cdb o\cdx[2]{o}{a1}{160}
\ocdx[2]{o}{a2}{120}{U_{2}}{above}
\ocdx[2]{o}{a3}{60}{U_{n}}{above}
\ocdx[2]{o}{a4}{20}{U_{n+1}}{right}
\node at (140:\xch) {$y_1$};
\node at (90:\xch) {$\cdots$};
\node at (40:\xch) {$y_n$};
}\right) \circ_{\beta} V\\
&=\ \left(\begin{array}{c}
 (T \succ_{\alpha \rhd \beta_1} U_1) \succ_{\alpha \rightarrow \beta_1} \treeoo{\cdb o\cdx[2]{o}{a1}{160}
\ocdx[2]{o}{a2}{120}{U_{2}}{above}
\ocdx[2]{o}{a3}{60}{U_{n}}{above}
\ocdx[2]{o}{a4}{20}{U_{n+1}}{right}
\node at (140:\xch) {$y_1$};
\node at (90:\xch) {$\cdots$};
\node at (40:\xch) {$y_n$};
}+(T \prec_{\alpha \lhd \beta_1} U_1)\succ_{\alpha \leftarrow \beta_1} \treeoo{\cdb o\cdx[2]{o}{a1}{160}
\ocdx[2]{o}{a2}{120}{U_{2}}{above}
\ocdx[2]{o}{a3}{60}{U_{n}}{above}
\ocdx[2]{o}{a4}{20}{U_{n+1}}{right}
\node at (140:\xch) {$y_1$};
\node at (90:\xch) {$\cdots$};
\node at (40:\xch) {$y_n$};
}\\
 +(T \circ_{\alpha \ast \beta_1} U_1) \succ_{\alpha \cdot \beta_1} \treeoo{\cdb o\cdx[2]{o}{a1}{160}
\ocdx[2]{o}{a2}{120}{U_{2}}{above}
\ocdx[2]{o}{a3}{60}{U_{n}}{above}
\ocdx[2]{o}{a4}{20}{U_{n+1}}{right}
\node at (140:\xch) {$y_1$};
\node at (90:\xch) {$\cdots$};
\node at (40:\xch) {$y_n$};
} \end{array}\right) \circ_{\beta} V\\
&=\ (T \succ_{\alpha \rhd \beta_1} U_1) \succ_{\alpha \rightarrow \beta_1} \left(  \treeoo{\cdb o\cdx[2]{o}{a1}{160}
\ocdx[2]{o}{a2}{120}{U_{2}}{above}
\ocdx[2]{o}{a3}{60}{U_{n}}{above}
\ocdx[2]{o}{a4}{20}{U_{n+1}}{right}
\node at (140:\xch) {$y_1$};
\node at (90:\xch) {$\cdots$};
\node at (40:\xch) {$y_n$};
} \circ_{\beta} V\right)+(T \prec_{\alpha \lhd \beta_1} U_1)\succ_{\alpha \leftarrow \beta_1} \left(  \treeoo{\cdb o\cdx[2]{o}{a1}{160}
\ocdx[2]{o}{a2}{120}{U_{2}}{above}
\ocdx[2]{o}{a3}{60}{U_{n}}{above}
\ocdx[2]{o}{a4}{20}{U_{n+1}}{right}
\node at (140:\xch) {$y_1$};
\node at (90:\xch) {$\cdots$};
\node at (40:\xch) {$y_n$};
} \circ_{\beta} V\right)\\
&\ +(T \circ_{\alpha \ast \beta_1} U_1) \succ_{\alpha \cdot \beta_1} \left(  \treeoo{\cdb o\cdx[2]{o}{a1}{160}
\ocdx[2]{o}{a2}{120}{U_{2}}{above}
\ocdx[2]{o}{a3}{60}{U_{n}}{above}
\ocdx[2]{o}{a4}{20}{U_{n+1}}{right}
\node at (140:\xch) {$y_1$};
\node at (90:\xch) {$\cdots$};
\node at (40:\xch) {$y_n$};
} \circ_{\beta} V\right) \hspace{1.5cm} \text{(by Case 8)}\\
&=\ T \succ_{\alpha} \left(U_1 \succ_{\beta_1} \left(  \treeoo{\cdb o\cdx[2]{o}{a1}{160}
\ocdx[2]{o}{a2}{120}{U_{2}}{above}
\ocdx[2]{o}{a3}{60}{U_{n}}{above}
\ocdx[2]{o}{a4}{20}{U_{n+1}}{right}
\node at (140:\xch) {$y_1$};
\node at (90:\xch) {$\cdots$};
\node at (40:\xch) {$y_n$};
} \circ_{\beta} V \right) \right)  \hspace{1.9cm} \text{(by Eq.~(\ref{eq:tri3}))}\\
&=\ T \succ_{\alpha} \left( \left(U_1 \succ_{\beta_1} \treeoo{\cdb o\cdx[2]{o}{a1}{160}
\ocdx[2]{o}{a2}{120}{U_{2}}{above}
\ocdx[2]{o}{a3}{60}{U_{n}}{above}
\ocdx[2]{o}{a4}{20}{U_{n+1}}{right}
\node at (140:\xch) {$y_1$};
\node at (90:\xch) {$\cdots$};
\node at (40:\xch) {$y_n$};
} \right) \circ_{\beta} V  \right)  \hspace{1.9cm}\text{(by the induction hypothesis)}\\
&=\ T \succ_{\alpha} (U \circ_{\beta} V).
\end{align*}
Hence Eq.~(\ref{eq:tri4}) holds. Now we show Eq.~(\ref{eq:tri5}) holds. If $T=\stree x$, then $(T \prec_{\alpha} U) \circ_{\beta} V=T \circ_{\beta} (U \succ_{\alpha} V)$ by the definition of $\circ_{\omega}$ in Case 6. For general case, we assume $l(T_{l+1})=\alpha_{l+1}$ if $T_{l+1} \neq |$, then
\begin{align*}
&\ (T \prec_{\alpha} U) \circ_{\beta} V\\
&=\ \left( \treeoo{\cdb o\ocdx[2]{o}{a1}{160}{T_{1}}{left}
\ocdx[2]{o}{a2}{120}{T_{2}}{above}
\ocdx[2]{o}{a3}{60}{T_{m}}{above}
\cdx[2]{o}{a4}{20}
\node at (140:\xch) {$x_1$};
\node at (90:\xch) {$\cdots$};
\node at (40:\xch) {$x_m$};
} \left(\begin{array}{c}
\prescript{\alpha_{m+1} \rightarrow \alpha}{}{T_{m+1}} \succ_{\alpha_{m+1} \rhd \alpha} U+\prescript{\alpha_{m+1} \leftarrow \alpha}{}{T_{m+1}} \prec_{\alpha_{m+1} \lhd \alpha} U\\
\ +\prescript{\alpha_{m+1} \cdot \alpha}{}{T_{m+1}} \circ_{\alpha_{m+1} \ast \alpha} U
\end{array}\right)\right) \circ_{\beta} V\\
&=\ \left(\begin{array}{c}
\treeoo{\cdb o\ocdx[2]{o}{a1}{160}{T_{1}}{left}
\ocdx[2]{o}{a2}{120}{T_{2}}{above}
\ocdx[2]{o}{a3}{60}{T_{m}}{above}
\cdx[2]{o}{a4}{20}
\node at (140:\xch) {$x_1$};
\node at (90:\xch) {$\cdots$};
\node at (40:\xch) {$x_m$};
} \prec_{\alpha_{l+1} \rightarrow \alpha} (T_{m+1} \succ_{\alpha_{m+1} \rhd \alpha} U)+\treeoo{\cdb o\ocdx[2]{o}{a1}{160}{T_{1}}{left}
\ocdx[2]{o}{a2}{120}{T_{2}}{above}
\ocdx[2]{o}{a3}{60}{T_{m}}{above}
\cdx[2]{o}{a4}{20}
\node at (140:\xch) {$x_1$};
\node at (90:\xch) {$\cdots$};
\node at (40:\xch) {$x_m$};
} \prec_{\alpha_{l+1} \leftarrow \alpha} (T_{m+1} \prec_{\alpha_{m+1} \lhd \alpha} U)\\
\ +\treeoo{\cdb o\ocdx[2]{o}{a1}{160}{T_{1}}{left}
\ocdx[2]{o}{a2}{120}{T_{2}}{above}
\ocdx[2]{o}{a3}{60}{T_{m}}{above}
\cdx[2]{o}{a4}{20}
\node at (140:\xch) {$x_1$};
\node at (90:\xch) {$\cdots$};
\node at (40:\xch) {$x_m$};
} \prec_{\alpha_{l+1} \cdot \alpha}(T_{m+1} \circ_{\alpha_{m+1} \ast \alpha} U)
\end{array} \right) \circ_{\beta} V.
\end{align*}
Next we consider the form of $\treeoo{\cdb o\ocdx[2]{o}{a1}{160}{T_{1}}{left}
\ocdx[2]{o}{a2}{120}{T_{2}}{above}
\ocdx[2]{o}{a3}{60}{T_{m}}{above}
\cdx[2]{o}{a4}{20}
\node at (140:\xch) {$x_1$};
\node at (90:\xch) {$\cdots$};
\node at (40:\xch) {$x_m$};
}$ as following:
\begin{enumerate}
\item If $\treeoo{\cdb o\ocdx[2]{o}{a1}{160}{T_{1}}{left}
\ocdx[2]{o}{a2}{120}{T_{2}}{above}
\ocdx[2]{o}{a3}{60}{T_{m}}{above}
\cdx[2]{o}{a4}{20}
\node at (140:\xch) {$x_1$};
\node at (90:\xch) {$\cdots$};
\node at (40:\xch) {$x_m$};
}=\stree x$ for some $x \in X$, then
\begin{align*}
&\left(\begin{array}{c}
 \stree x \prec_{\alpha_{l+1} \rightarrow \alpha} (T_{m+1} \succ_{\alpha_{m+1} \rhd \alpha} U)+\stree x \prec_{\alpha_{l+1} \leftarrow \alpha} (T_{m+1} \prec_{\alpha_{m+1} \lhd \alpha} U)\\
 \ +\stree x \prec_{\alpha_{l+1} \cdot \alpha}(T_{m+1} \circ_{\alpha_{m+1} \ast \alpha} U)
 \end{array}\right) \circ_{\beta} V\\
&=\ \stree x \circ_{\beta} \left(\begin{array}{c}
(T_{l+1} \succ_{\alpha_{l+1} \rhd \alpha} U) \succ_{\alpha_{l+1} \rightarrow \alpha} V+(T_{l+1} \prec_{\alpha_{l+1} \lhd \alpha} U) \succ_{\alpha_{l+1} \leftarrow \alpha} V\\
\ + (T_{l+1} \circ_{\alpha_{l+1} \ast \alpha} U) \succ_{\alpha_{l+1} \cdot \alpha} V
\end{array}\right) && \text{(by Case 6)}\\
&=\ \stree x \circ_{\beta} \big(T_{l+1} \succ_{\alpha_{l+1}} (U \succ_{\alpha} V) \big) &&\text{(by Eq.~(\ref{eq:tri3}))}\\
&=\ \left(\stree x \prec_{\alpha_{l+1}} T_{l+1}\right) \circ_{\beta} (U \succ_{\alpha} V) &&\text{(by Case 6)}\\
&=\ T \circ_{\beta} (U \succ_{\alpha} V).
\end{align*}

\item If there are $T',T'' \in \frat(X,\Omega)$ and $\omega \in \Omega$ such that $\treeoo{\cdb o\ocdx[2]{o}{a1}{160}{T_{1}}{left}
\ocdx[2]{o}{a2}{120}{T_{2}}{above}
\ocdx[2]{o}{a3}{60}{T_{m}}{above}
\cdx[2]{o}{a4}{20}
\node at (140:\xch) {$x_1$};
\node at (90:\xch) {$\cdots$};
\node at (40:\xch) {$x_m$};
}=T' \succ_{\omega} T''$, then
\begin{align*}
&\ \left(\begin{array}{c}
(T' \succ_{\omega} T'') \prec_{\alpha_{l+1} \rightarrow \alpha} (T_{m+1} \succ_{\alpha_{m+1} \rhd \alpha} U)\\
+(T' \succ_{\omega} T'') \prec_{\alpha_{l+1} \leftarrow \alpha} (T_{m+1} \prec_{\alpha_{m+1} \lhd \alpha} U)\\
 \ +(T' \succ_{\omega} T'') \prec_{\alpha_{l+1} \cdot \alpha}(T_{m+1} \circ_{\alpha_{m+1} \ast \alpha} U)
 \end{array}\right)\circ_{\beta} V\\
&=\ \left(\begin{array}{c}
 T' \succ_{\omega} \big( T'' \prec_{\alpha_{l+1} \rightarrow \alpha} (T_{m+1} \succ_{\alpha_{m+1} \rhd \alpha} U)\\
 + T''  \prec_{\alpha_{l+1} \leftarrow \alpha} (T_{m+1} \prec_{\alpha_{m+1} \lhd \alpha} U) \\
\ +T'' \prec_{\alpha_{l+1} \cdot \alpha}(T_{m+1} \circ_{\alpha_{m+1} \ast \alpha} U)\big)
\end{array}\right) \circ_{\beta} V && \text{(by Eq.~(\ref{eq:tri2}))}\\
&=\ \big( T' \succ_{\omega} ((T'' \prec_{\alpha_{l+1}} T_{l+1}) \prec_{\alpha} U) \big) \circ_{\beta} V && \text{(by Eq.~(\ref{eq:tri1}))}\\
&=\ T' \succ_{\omega} \big(((T'' \prec_{\alpha_{l+1}} T_{l+1}) \prec_{\alpha} U) \circ_{\beta} V \big) && \text{(by Case 8)}\\
&=\ T' \succ_{\omega} \big((T'' \prec_{\alpha_{l+1}} T_{l+1}) \circ_{\beta} (U \succ_{\alpha} V) \big) &&\text{(by induction hypothesis)}\\
&=\ \big( T' \succ_{\omega} (T''\prec_{\alpha_{l+1}} T_{l+1}) \big) \circ_{\beta}(U \succ_{\alpha} V) &&\text{(by Case 8)}\\
&=\ \big((T' \succ_{\omega} T'') \prec_{\alpha_{l+1}} T_{l+1} \big) \circ_{\beta}(U \succ_{\alpha} V)&&
\text{(by Eq.~(\ref{eq:tri2})}\\
&=\ T \circ_{\beta} (U \succ_{\alpha} V).
\end{align*}

\item If there are $\stree x,T'' \in \frat(X,\Omega)$ and $\omega \in \Omega$ such that $\treeoo{\cdb o\ocdx[2]{o}{a1}{160}{T_{1}}{left}
\ocdx[2]{o}{a2}{120}{T_{2}}{above}
\ocdx[2]{o}{a3}{60}{T_{m}}{above}
\cdx[2]{o}{a4}{20}
\node at (140:\xch) {$x_1$};
\node at (90:\xch) {$\cdots$};
\node at (40:\xch) {$x_m$};
}=\stree x \circ_{\omega} T''$, then
\begin{align*}
&\left( \begin{array}{c}
\left(\stree x \circ_{\omega} T''\right) \prec_{\alpha_{l+1} \rightarrow \alpha} (T_{m+1} \succ_{\alpha_{m+1} \rhd \alpha} U)\\
+\left(\stree x \circ_{\omega} T''\right) \prec_{\alpha_{l+1} \leftarrow \alpha} (T_{m+1} \prec_{\alpha_{m+1} \lhd \alpha} U)\\
 \ +\left(\stree x \circ_{\omega} T''\right) \prec_{\alpha_{l+1} \cdot \alpha}(T_{m+1} \circ_{\alpha_{m+1} \ast \alpha} U)
 \end{array}\right) \circ_{\beta} V\\
&=\ \left(\begin{array}{c}
\stree x \circ_{\omega} \left(\begin{array}{c}
T'' \prec_{\alpha_{l+1} \rightarrow \alpha} (T_{m+1} \succ_{\alpha_{m+1} \rhd \alpha} U)\\
+ T''  \prec_{\alpha_{l+1} \leftarrow \alpha} (T_{m+1} \prec_{\alpha_{m+1} \lhd \alpha} U) \\
 +T'' \prec_{\alpha_{l+1} \cdot \alpha}(T_{m+1} \circ_{\alpha_{m+1} \ast \alpha} U)
\end{array} \right)
\end{array}\right)\circ_{\beta} V &&\text{(by Case 7)}\\
&=\ \left( \stree x \circ_{\omega} ((T'' \prec_{\alpha_{l+1}} T_{l+1}) \prec_{\alpha} U) \right) \circ_{\beta} V &&\text{(by Eq.~(\ref{eq:tri1}))}\\
&=\ \stree x \circ_{\omega} \big(((T'' \prec_{\alpha_{l+1}} T_{l+1}) \prec_{\alpha} U) \circ_{\beta} V \big) &&\text{(by Case 7)}\\
&=\ \stree x \circ_{\omega} \big((T'' \prec_{\alpha_{l+1}} T_{l+1}) \circ_{\beta} (U \succ_{\alpha} V) \big) && \text{(by induction hypothesis)}\\
&=\ \left( \stree x \circ_{\omega} (T''\prec_{\alpha_{l+1}} T_{l+1}) \right) \circ_{\beta}(U \succ_{\alpha} V) &&\text{(by Case 7)}\\
&=\ \left(\left(\stree x \circ_{\omega} T''\right) \prec_{\alpha_{l+1}} T_{l+1} \right) \circ_{\beta}(U \succ_{\alpha} V)
 &&\text{(by induction hypothesis and Eq.~(\mref{eq:tri6}))}\\
&=\ T \circ_{\beta} (U \succ_{\alpha} V).
\end{align*}
\end{enumerate}
Hence Eq.~(\ref{eq:tri5}) holds.  Eq.~(\ref{eq:tri6}) holds directly as $T \circ_{\alpha} U$ does not change the rightmost branch of $U$ and $U \prec_{\beta} V$ only changes the rightmost branch of $U$ and $U$ has at least two branches. Finally, we show Eq.~(\ref{eq:tri7}) holds by induction on $\leaf(T)$. If $T=\stree x$ for some $x \in X$, then $(T \circ_{\alpha} U) \circ_{\beta} V=T \circ_{\alpha} (U \circ_{\beta} V)$ by the definition of $\circ_{\omega}$ in Case 7. Suppose Eq.~(\ref{eq:tri7}) holds for all $T$, where $\leaf(T) \leq q$ with $q$ a fixed integer. Assume $\leaf(T)=p+1$, we consider the form of $T$ as follows.
\begin{enumerate}
\item If there are $T', T'' \in \frat(X,\Omega)$ and $\omega \in \Omega$ such that $T=T' \prec_{\omega} T''$, then
\begin{align*}
&\ (T \circ_{\alpha} U) \circ_{\beta} V=\big((T' \prec_{\omega} T'')\circ_{\alpha} U\big) \circ_{\beta} V\\
&=\ \big(T' \circ_{\alpha} (T'' \succ_{\omega} U) \big) \circ_{\beta} V && \text{(by Eq.~{\ref{eq:tri5}})}\\
&=\ T' \circ_{\alpha} \big((T'' \succ_{\omega} U) \circ_{\beta} V \big) && \text{(by induction hypothesis)}\\
&=\ T' \circ_{\alpha} \big(T'' \succ_{\omega} (U \circ_{\beta} V) \big) && \text{(by Eq.~(\ref{eq:tri4}))}\\
&=\ (T' \prec_{\omega} T'') \circ_{\alpha} (U \circ_{\beta} V) && \text{(by Case 6)}\\
&=\ T \circ_{\alpha} (U \circ_{\beta} V).
\end{align*}

\item If there are $T', T'' \in \frat(X,\Omega)$ and $\omega\in \Omega$ such that $T=T' \succ_{\omega} T''$, then
\begin{align*}
&\ (T\circ_{\alpha} U)\circ_{\beta} V=\big((T'\succ_{\omega} T'') \circ_{\alpha} U\big) \circ_{\beta} V\\
&=\ \big( T' \succ_{\omega} (T'' \circ_{\alpha} U) \big) \circ_{\beta} V && \text{(by Case 8)}\\
&=\ T' \succ_{\omega} \big((T'' \circ_{\alpha} U) \circ_{\beta} V \big) && \text{(by Case 8)}\\
&=\ T' \succ_{\omega} \big( T'' \circ_{\alpha} (U \circ_{\beta} V) \big) && \text{(by induction hypothesis)}\\
&=\ (T' \succ_{\omega} T'') \circ_{\alpha} (U \circ_{\beta} V) && \text{(by Case 8)}\\
&=\ T \circ_{\alpha} (U \circ_{\beta} V).
\end{align*}

\item If there are $\stree x, T'' \in \frat(X,\Omega)$ and $\omega \in \Omega$ such that $T=\stree x \circ_{\omega} T''$, then
\begin{align*}
(T \circ_{\alpha} U) \circ_{\beta} V&=\left(\left(\stree x \circ_{\omega} T''\right) \circ_{\alpha} U\right) \circ_{\beta} V\\
&= \left( \stree x \circ_{\omega} (T'' \circ_{\alpha} U) \right) \circ_{\beta} V && \text{( by Case 7)}\\
&= \stree x \circ_{\omega} \big((T'' \circ_{\alpha} U) \circ_{\beta} V\big) && \text{( by Case 7)}\\
&= \stree x \circ_{\omega} \big( T''\circ_{\alpha} (U \circ_{\beta} V)\big) && \text{( by induction hypothesis)}\\
&= \left(\stree x \circ_{\omega} T''\right) \circ_{\alpha} (U \circ_{\beta} V) && \text{(by Case 7)}\\
&= T \circ_{\alpha} (U \circ_{\beta} V).
\end{align*}
\end{enumerate}
Hence Eq.~(\ref{eq:tri7}) holds. So $\bfk \frat(X,\Omega)$ is an $\Omega$-tridendriform algebra.\\

Let $(A,(\prec_{\omega}, \succ_{\omega}, \circ_{\omega})_{\omega \in \Omega})$ be an $\Omega$-tridendriform algebra and $f:X \rightarrow A$ a set map. We extend $f$ to be an $\Omega$-tridendriform algebra $\overline{f}: \bfk \frat(X,\Omega) \rightarrow A$ such that $\overline{f} \circ j=f$. For $T \in \frat(X,\Omega)$ with $\leaf(T)=2$, i.e. $T=\stree x$ for some $x \in X$, define $\overline{f}(T)=f(x)$. Suppose $\overline{f}(T)$ has been defined for all $T$ with $\leaf(T) \leq q$, where $q \geq 2$ is a fixed integer. Consider the case of $\leaf(T)=q+1$. We consider the form of $T$ as follows.
\begin{enumerate}
\item  If $T=\treeoo{\cdb o
\cdx[2]{o}{a2}{120}
\ocdx[2]{o}{a3}{60}{\prescript{\alpha}{}{T_2}}{above}
\node at (90:\xch) {$x$};
}$, then define
\begin{align*}
\overline{f}(T) :=f(x) \prec_{\alpha} \overline{f}(T_2).
\end{align*}

\item If $T= \treeoo{\cdb o\cdx[2]{o}{a1}{160}
\ocdx[2]{o}{a2}{120}{\prescript{\alpha}{}{T_2}}{above}
\ocdx[2]{o}{a3}{60}{T_{m}}{above}
\ocdx[2]{o}{a4}{20}{T_{m+1}}{right}
\node at (140:\xch) {$x_1$};
\node at (90:\xch) {$\cdots$};
\node at (40:\xch) {$x_m$};
}$ with $m \geq 2$, then define
\begin{align*}
\overline{f}(T):&=\ f(x) \circ_{\alpha} \overline{f}\left(\treeoo{\cdb o
\ocdx[2]{o}{a2}{135}{T_2}{above}
\ocdx[2]{o}{a3}{90}{T_{m}}{above}
\ocdx[2]{o}{a4}{45}{T_{m+1}}{right}
\node at (110:\xch) {$\cdots$};
\node at (70:\xch) {$x_m$};
}\right).
\end{align*}

\item If $T=\treeoo{\cdb o\ocdx[2]{o}{a1}{160}{T_1^{\alpha}}{left}
\ocdx[2]{o}{a2}{120}{T_2}{above}
\ocdx[2]{o}{a3}{60}{T_m}{above}
\ocdx[2]{o}{a4}{20}{T_{m+1}}{right}
\node at (140:\xch) {$x_1$};
\node at (90:\xch) {$\cdots$};
\node at (40:\xch) {$x_m$};
}$ with $T_1 \neq |$, then define
\begin{align*}
\overline{f}(T) :=\overline{f}(T_1) \succ_{\alpha} \overline{f}\left(\treeoo{\cdb o\cdx[2]{o}{a1}{160}
\ocdx[2]{o}{a2}{120}{T_2}{above}
\ocdx[2]{o}{a3}{60}{T_{m}}{above}
\ocdx[2]{o}{a4}{20}{T_{m+1}}{right}
\node at (140:\xch) {$x_1$};
\node at (90:\xch) {$\cdots$};
\node at (40:\xch) {$x_m$};
}\right).
\end{align*}
\end{enumerate}
We can get that it is the unique way to extend $f$ as an $\Omega$-tridendriform algebra morphism. Hence,
 $\bfk \frat(X,\Omega)$ and the map $j$ is the free $\Omega$-tridendriform algebra generated by $X$.
\end{proof}

\subsection{Commutative $\Omega$-tridendriform algebras on typed words}

Let us first recall the concept of associative matching algebras \cite{ZGG20}.

\begin{defn} \label{defmatch}
An {\bf associative matching algebra} is a tuple $(A, (\star_{\omega})_{\omega \in \Omega})$, where $A$ is a vector space and for each $\omega \in \Omega$, $\star_{\omega}: A \ot A \rightarrow A$ is a linear map such that
\begin{align*}
(a \star_{\alpha} b) \star_{\beta} c= a \star_{\alpha} (b \star_{\beta} c)
\end{align*}
for all $a,b,c \in A$ and $\alpha,\beta \in \Omega$.\\
\end{defn}

As in \cite{FoissyPeng}, the space of $\Omega$-typed words in $A$ is
\begin{align*}
\shh(A)=\bigoplus_{n \geq 1} \underbrace{A \ot (\bfk \Omega) \ot \cdots \ot (\bfk \Omega) \ot A}_{\text{$n$'s $A$ and $(n-1)$'s
$(\bfk \Omega)$}}.
\end{align*}
For ease of statement, we shall write each pure tensor $\mathbf{v}=v_0 \ot \omega_1 \ot \cdots \ot \omega_n \ot v_n \in \Omega$
under  the form
\begin{align*}
\mathbf{v}=v_0 \ot_{\omega_1} v_1 \ot_{\omega_2} \cdots \ot_{\omega_n} v_n,
\end{align*}
where $n \geq 0$, $\omega_1, \cdots, \omega_n \in \Omega$ and $v_0, \cdots, v_n \in V$ with the convention $\mathbf{v}=v_0$ if $n=0$. We call $\mathbf{v}$ an {\bf $\Omega$-typed word} in $V$ and define its {\bf length} $\ell(\mathbf{v}):=n+1$.

Let $\Omega$ be a set with six products $\leftarrow, \rightarrow, \lhd, \rhd, \cdot, \ast$. For $\omega \in \Omega$, we define products $\prec_{\omega}, \succ_{\omega}, \circ_{\omega}$ on $\shh(A)$ recursively in the following way:

For $\mathbf{a},\mathbf{b} \in \shh(A)$, if $\ell(\mathbf{a})+\ell(\mathbf{b})=2$, then
\begin{align*}
\ell(\mathbf{a})=\ell(\mathbf{b})=1, \,\,  \mathbf{a}=a \,\, \text{ and }\,\, \mathbf{b}=b\,\, \text{ where } a,b \in A.
\end{align*}
Define
\begin{align*}
\mathbf{a} \prec_{\omega} \mathbf{b}:= a \ot_{\omega} b, \,\, \mathbf{a} \circ_{\omega} \mathbf{b}=a \star_{\omega} b \,\, \text{ and }\,\, \mathbf{a} \succ_{\omega} \mathbf{b}= b \ot_{\omega} a.
\end{align*}

For the induction step of $\ell(\mathbf{a})+\ell(\mathbf{b}) >1$, to define $\mathbf{a} \prec_{\omega} \mathbf{b}$, we consider the following two cases.

If $\ell(\mathbf{a})=1$ and $\mathbf{a}=a_1 \in A$, then define
\begin{align*}
\mathbf{a} \prec_{\omega} \mathbf{b} :=a_1 \ot_{\omega} \mathbf{b}.
\end{align*}
Otherwise $\ell(\mathbf{a}) \geq 2$ and write $\mathbf{a}=a_1 \ot_{\alpha_1} \mathbf{a}'$ where $a_1 \in A$, then define
\begin{align*}
\mathbf{a} \prec_{\omega} \mathbf{b} :&=\ a_1 \ot_{\alpha_1 \rightarrow \omega} (\mathbf{a}' \succ_{\alpha_1 \rhd \omega} \mathbf{b})
+a_1 \ot_{\alpha_1 \leftarrow \omega} (\mathbf{a}' \prec_{\alpha_1 \lhd \omega} \mathbf{b}) +a_1 \ot_{\alpha_1 \cdot \omega} (\mathbf{a}' \circ_{\alpha_1 \ast \omega} \mathbf{b}).
\end{align*}

To define $\mathbf{a} \succ_{\omega} \mathbf{b}$, we consider the following two cases.

If $\ell(\mathbf{b})=1$ and $\mathbf{b}=b_1 \in A$, then define
\begin{align*}
\mathbf{a} \succ_{\omega} \mathbf{b} :=b_1 \ot_{\omega} \mathbf{a}.
\end{align*}
Otherwise $\ell(\mathbf{b}) \geq 2$ and write $\mathbf{b}=b_1 \ot_{\beta_1} \mathbf{b}'$ where $b_1 \in A$, then define
\begin{align*}
\mathbf{a} \succ_{\omega} \mathbf{b} := b_1 \ot_{\omega \rightarrow \beta_1} (\mathbf{a} \succ_{\omega \rhd \beta_1} \mathbf{b}')
+b_1 \ot_{\omega \leftarrow \beta_1} (\mathbf{a} \prec_{\omega \lhd \beta_1} \mathbf{b}')
+b_1 \ot_{\omega \cdot \beta_1} (\mathbf{a} \circ_{\omega \ast \beta_1} \mathbf{b}').
\end{align*}

To define $\mathbf{a} \circ_{\omega} \mathbf{b}$, we consider the following three cases.

If $\ell(\mathbf{a})=1$ and $\ell(\mathbf{b}) \geq 2$, write $\mathbf{a}=a_1$ and $\mathbf{b}=b_1 \ot_{\beta_1} \mathbf{b}'$ where $a_1, b_1 \in A$, then define
\begin{align*}
\mathbf{a} \circ_{\omega} \mathbf{b} :=(a_1 \star_{\omega} b_1) \ot_{\beta_1} \mathbf{b}'.
\end{align*}
Otherwise, if $\ell(\mathbf{a}) \geq 2$ and $\ell(\mathbf{b})=1$, write $\mathbf{a}=a_1 \ot_{\alpha_1} \mathbf{a}'$ and $\mathbf{b}=b_1$ where $a_1,b_1 \in A$, then define
\begin{align*}
\mathbf{a} \circ_{\omega} \mathbf{b} :=(a_1 \star_{\omega} b_1) \ot_{\alpha_1} \mathbf{a}'.
\end{align*}
Otherwise, $\mathbf{a} \geq 2$ and $\mathbf{b} \geq 2$, write $\mathbf{a}=a_1 \ot_{\alpha_1} \mathbf{a}'$ and $\mathbf{b}=b_1 \ot_{\beta_1} \mathbf{b}'$ where $a_1, b_1 \in A$, then define
\begin{align*}
\mathbf{a} \circ_{\omega} \mathbf{b} :&=\ (a _1 \star_{\omega} b_1) \ot_{\alpha_1 \rightarrow \beta_1} (\mathbf{a}' \succ_{\alpha_1 \rhd \beta_1} \mathbf{b}')
+(a_1 \star_{\omega} b_1)\ot_{\alpha_1 \leftarrow \beta_1} (\mathbf{a}' \prec_{\alpha_1 \lhd \beta_1} \mathbf{b}')\\
&\ +(a_1 \star_{\omega} b_1)\ot_{\alpha_1 \cdot \beta_1} (\mathbf{a}' \circ_{\alpha_1 \ast \beta_1} \mathbf{b}').
\end{align*}

We obtain the following result:

\begin{theorem} \label{theocomtridend}
With these products defined as above, if $(\Omega, \leftarrow, \rightarrow, \lhd, \rhd, \cdot, \ast)$ is an ETS, then $\shh(A)$ is an $\Omega$-tridendriform algebra. Moreover, if $\Omega$ is commutative and if
 $(A,(\star_\omega)_{\omega \in \Omega})$ is commutative, that is to say: for any $a,b\in A$,
 for any $\omega \in \Omega$, $a\star_\omega b=b\star_\omega a$, then $\shh(A)$ is a commutative $\Omega$-tridendriform algebra.
\end{theorem}

\begin{proof}
Denote by $1$ the empty $A$-typed word and $\sh(A):=\bfk 1 \oplus \shh(A)$. Extending the products $\prec_{\omega}, \succ_{\omega}, \circ_{\omega}$ to the space $\sh(A) \ot \shh(A) \oplus \shh(A) \ot \sh(A)$ by
\begin{align*}
1 \succ_\omega \mathbf{a}&:=\mathbf{a} \prec_\omega 1:=\mathbf{a},&
1 \prec_\omega \mathbf{a}&:=\mathbf{a} \succ_\omega 1:=0&& \text{ and }&
 1\circ_{\omega} \mathbf{a}&:=\mathbf{a} \circ_{\omega} 1:=0,
\end{align*}
for all $\omega\in\Omega$ and $\mathbf{a} \in \shh(A)$.
The products $\prec_{\omega}, \succ_{\omega}, \circ_{\omega}$ can now  be rewritten in the following way: for $\mathbf{a}=a_1 \ot_{\alpha_1} \mathbf{a}', \mathbf{b}=b_1 \ot_{\beta_1} \mathbf{b}' \in \shh(A)$.
By convention, we add an element $\emptyset$ to $\Omega$, as a unit for the six products of $\Omega$.
Note that $\mathbf{a}=a_1$ and $\alpha_1=\emptyset$ if $\ell(\mathbf{a})=1$; $\mathbf{b}=b_1$ and $\beta_1=\emptyset$ if $\ell(\mathbf{b})=1$.
\begin{align*}
\mathbf{a} \prec_{\omega} \mathbf{b} :&=\ a_1 \ot_{\alpha_1 \rightarrow \omega} (\mathbf{a}' \succ_{\alpha_1 \rhd \omega} \mathbf{b})
+a_1 \ot_{\alpha_1 \leftarrow \omega} (\mathbf{a}' \prec_{\alpha_1 \lhd \omega} \mathbf{b}) +a_1 \ot_{\alpha_1 \cdot \omega} (\mathbf{a}' \circ_{\alpha_1 \ast \omega} \mathbf{b}),\\
\mathbf{a} \succ_{\omega} \mathbf{b} :&=\ b_1 \ot_{\omega \rightarrow \beta_1} (\mathbf{a} \succ_{\omega \rhd \beta_1} \mathbf{b}')
+b_1 \ot_{\omega \leftarrow \beta_1} (\mathbf{a} \prec_{\omega \lhd \beta_1} \mathbf{b}')
+b_1 \ot_{\omega \cdot \beta_1} (\mathbf{a} \circ_{\omega \ast \beta_1} \mathbf{b}'),\\
\mathbf{a} \circ_{\omega} \mathbf{b} :&=\ (a _1 \star_{\omega} b_1) \ot_{\alpha_1 \rightarrow \beta_1} (\mathbf{a}' \succ_{\alpha_1 \rhd \beta_1} \mathbf{b}')
+(a_1 \star_{\omega} b_1)\ot_{\alpha_1 \leftarrow \beta_1} (\mathbf{a}' \prec_{\alpha_1 \lhd \beta_1} \mathbf{b}')\\
&\ +(a_1 \star_{\omega} b_1)\ot_{\alpha_1 \cdot \beta_1} (\mathbf{a}' \circ_{\alpha_1 \ast \beta_1} \mathbf{b}').
\end{align*}

Now we show that $\shh(A)$ is an $\Omega$-tridendriform algebra. For $\mathbf{a}, \mathbf{b}, \mathbf{c} \in \shh(A)$, we prove Eqs.~(\ref{eq:tri1})-(\ref{eq:tri7}) hold by induction on the sum $\ell(\mathbf{a})+\ell(\mathbf{b})+\ell(\mathbf{c})$.
If $\ell(\mathbf{a})+\ell(\mathbf{b})+\ell(\mathbf{c})=3$, then $\ell(\mathbf{a})=\ell(\mathbf{b})=\ell(\mathbf{c})=1$ and $\mathbf{a}=a_1, \mathbf{b}=b_1, \mathbf{c}=c_1 \in A$, Eqs.~(\ref{eq:tri1})-(\ref{eq:tri7}) hold by direct calculation.

For the induction step of $\ell(\mathbf{a})+\ell(\mathbf{b})+\ell(\mathbf{c}) \geq 4$, assume $\mathbf{a}=a_1 \ot_{\alpha_1} \mathbf{a}'$, then
\begin{align*}
&\ (\mathbf{a} \prec_{\alpha} \mathbf{b}) \prec_{\beta} \mathbf{c}=((a_1 \ot_{\alpha_1} \mathbf{a}') \prec_{\alpha} \mathbf{b}) \prec_{\beta} \mathbf{c}\\
&=\ \big(a_1 \ot_{\alpha_1 \rightarrow \alpha} (\mathbf{a}' \succ_{\alpha_1 \rhd \alpha} \mathbf{b})+ a_1 \ot_{\alpha_1 \leftarrow \alpha} (\mathbf{a}' \prec_{\alpha_1 \lhd \alpha} \mathbf{b})+ a_1 \ot_{\alpha_1 \cdot \alpha}(\mathbf{a}' \circ_{\alpha_1 \ast \alpha} \mathbf{b}) \big) \prec_{\beta} \mathbf{c}\\
&=\ a_1 \ot_{(\alpha_1 \rightarrow \alpha)\rightarrow \beta} \big( (\mathbf{a}' \succ_{\alpha_1 \rhd \alpha} \mathbf{b})\succ_{(\alpha_1 \rightarrow \alpha)\rhd \beta} \mathbf{c} \big)+a_1 \ot_{(\alpha_1 \rightarrow \alpha) \leftarrow \beta} \big((\mathbf{a}' \succ_{\alpha_1 \rhd \alpha} \mathbf{b})\prec_{(\alpha_1 \rightarrow \alpha) \lhd \beta} \mathbf{c} \big)\\
&\ + a_1\ot_{(\alpha_1 \rightarrow \alpha) \cdot \beta} \big((\mathbf{a}' \succ_{\alpha_1 \rhd \alpha} \mathbf{b}) \circ_{(\alpha_1 \rightarrow \alpha) \ast \beta} \mathbf{c} \big)+a_1 \ot_{(\alpha_1 \leftarrow \alpha) \rightarrow \beta} \big((\mathbf{a}' \prec_{\alpha_1 \lhd \alpha} \mathbf{b}) \succ_{(\alpha_1 \leftarrow \alpha) \rhd \beta} \mathbf{c} \big)\\
&\ +a_1 \ot_{(\alpha_1 \leftarrow \alpha) \leftarrow \beta} \big((\mathbf{a}' \prec_{\alpha_1 \lhd \alpha} \mathbf{b}) \prec_{(\alpha_1 \leftarrow \alpha) \lhd \beta} \mathbf{c} \big)+ a_1 \ot_{(\alpha_1 \leftarrow \alpha) \cdot \beta} \big((\mathbf{a}' \prec_{\alpha_1 \lhd \alpha} \mathbf{b}) \circ_{(\alpha_1 \leftarrow \alpha) \ast \beta} \mathbf{c} \big)\\
&\ +a_1 \ot_{(\alpha_1 \cdot \alpha)\rightarrow \beta} \big((\mathbf{a}' \circ_{\alpha_1 \ast \alpha} \mathbf{b})\succ_{(\alpha_1 \cdot \alpha)\rhd \beta} \mathbf{c} \big)+ a_1 \ot_{(\alpha_1 \cdot \alpha) \leftarrow \beta} \big((\mathbf{a}' \circ_{\alpha_1 \ast \alpha} \mathbf{b}) \prec_{(\alpha_1 \cdot \alpha)\lhd \beta} \mathbf{c} \big)\\
&\ +a_1 \ot_{(\alpha_1 \cdot \alpha) \cdot \beta} \big((\mathbf{a}' \circ_{\alpha_1 \ast \alpha} \mathbf{b})\circ_{(\alpha_1 \cdot \alpha) \ast \beta} \mathbf{c} \big)\\
&=\ a_1 \ot_{\alpha_1 \rightarrow (\alpha \rightarrow \beta)} \big(\mathbf{a}' \succ_{\alpha_1 \rhd (\alpha \rightarrow \beta)} (\mathbf{b} \succ_{\alpha \rhd \beta} \mathbf{c}) \big)+a_1 \ot_{\alpha_1 \leftarrow (\alpha \rightarrow \beta)} \big(\mathbf{a}' \prec_{\alpha_1 \lhd (\alpha \rightarrow \beta)} (\mathbf{b} \succ_{\alpha \rhd \beta} \mathbf{c}) \big)\\
&\ +a_1 \ot_{\alpha_1 \cdot (\alpha \rightarrow \beta)} \big(\mathbf{a}' \circ_{\alpha_1 \ast (\alpha \rightarrow \beta)} (\mathbf{b} \succ_{\alpha \rhd \beta} \mathbf{c}) \big)+a_1 \ot_{\alpha_1 \rightarrow (\alpha \leftarrow \beta)} \big(\mathbf{a}' \succ_{\alpha_1 \rhd (\alpha \leftarrow \beta)} (\mathbf{b} \prec_{\alpha \lhd \beta} \mathbf{c}) \big)\\
&\ +a_1 \ot_{\alpha_1 \leftarrow (\alpha \leftarrow \beta)} \big(\mathbf{a}' \prec_{\alpha_1 \lhd (\alpha \leftarrow \beta)}(\mathbf{b} \prec_{\alpha \lhd \beta} \mathbf{c}) \big)+a_1 \ot_{\alpha_1 \cdot (\alpha \leftarrow \beta)} \big(\mathbf{a}' \circ_{\alpha_1 \ast (\alpha \leftarrow \beta)} (\mathbf{b} \prec_{\alpha \lhd \beta} \mathbf{c}) \big)\\
&\ +a_1 \ot_{\alpha_1 \rightarrow (\alpha \cdot \beta)} \big(\mathbf{a}' \succ_{\alpha_1 \rhd (\alpha \cdot \beta)} (\mathbf{b} \circ_{\alpha \ast \beta} \mathbf{c}) \big)+a_1 \ot_{\alpha_1 \leftarrow(\alpha \cdot \beta)} \big(\mathbf{a}' \prec_{\alpha_1 \lhd (\alpha \cdot \beta)} (\mathbf{b} \circ_{\alpha \ast \beta} \mathbf{c}) \big)\\
&\ +a_1 \ot_{\alpha_1 \cdot (\alpha \cdot \beta)} \big(\mathbf{a}' \circ_{\alpha_1 \ast (\alpha \cdot \beta)} (\mathbf{b} \circ_{\alpha \ast \beta} \mathbf{c}) \big)
\hspace{2cm}\text{(by induction hypothesis and $\Omega$ being an ETS)}\\
&=\ \mathbf{a} \prec_{\alpha \rightarrow \beta}(\mathbf{b} \succ_{\alpha \rightarrow \beta} \mathbf{c})+\mathbf{a} \prec_{\alpha \leftarrow \beta}(\mathbf{b} \prec_{\alpha \lhd \beta} \mathbf{c})+ \mathbf{a} \prec_{\alpha \cdot \beta}(\mathbf{b} \circ_{\alpha \ast \beta} \mathbf{c}).
\end{align*}
Hence Eq.~(\ref{eq:tri1}) holds. Similarly, it can be proved that Eqs.~(\ref{eq:tri2})-(\ref{eq:tri7}) hold. So $\shh(A)$ is an $\Omega$-tridendriform algebra.

Next, assume $\Omega$ is commutative and $(A, (\star_{\omega})_{\omega \in \Omega})$ is commutative, we prove that
\begin{align*}
\mathbf{a} \prec_{\alpha} \mathbf{b}&=\mathbf{b} \succ_{\alpha} \mathbf{a} &&\text{ and }
&\mathbf{a} \circ_{\alpha} \mathbf{b}&=\mathbf{b} \circ_{\alpha} \mathbf{a}
\end{align*}
for all $\mathbf{a}, \mathbf{b} \in \shh(A)$ and $\alpha \in \Omega$ by induction on $\ell(\mathbf{a})+\ell(\mathbf{b})$. If $\ell(\mathbf{a})+\ell(\mathbf{b})=2$, then $\ell(\mathbf{a})=\ell(\mathbf{b})=1$ and $\mathbf{a}=a,\mathbf{b}=b \in A$. So
\begin{align*}
\mathbf{a} \prec_{\alpha} \mathbf{b}&=a \ot_{\alpha} b=\mathbf{b} \succ_{\alpha} \mathbf{a}&&\text{ and }& \mathbf{a} \circ_{\alpha} \mathbf{b}&=a \star_{\alpha} b=b \star_{\alpha} a=\mathbf{b} \circ_{\alpha} \mathbf{a}.
\end{align*}
For the inductive step of $\ell(\mathbf{a})+\ell(\mathbf{b}) \geq 3$, assume $\mathbf{a}=a_1 \ot_{\alpha_1} \mathbf{a}'$, then
\begin{align*}
 \mathbf{a} \prec_{\alpha} \mathbf{b}
&=\ a_1 \ot_{\alpha_1 \rightarrow \alpha} (\mathbf{a}' \succ_{\alpha_1 \rhd \alpha} \mathbf{b})+ a_1 \ot_{\alpha_1 \leftarrow \alpha}(\mathbf{a}' \prec_{\alpha_1 \lhd \alpha} \mathbf{b})+a_1 \ot_{\alpha_1 \cdot \alpha} (\mathbf{a}' \circ_{\alpha_1 \ast \alpha} \mathbf{b})\\
&=\ a_1 \ot_{\alpha \leftarrow \alpha_1} (\mathbf{b} \prec_{\alpha \lhd \alpha_1} \mathbf{a}')+ a_1 \ot_{\alpha \rightarrow \alpha_1} (\mathbf{b} \succ_{\alpha \rhd \alpha_1} \mathbf{a}')+ a_1 \ot_{\alpha \cdot \alpha_1} (\mathbf{b} \circ_{\alpha \ast \alpha_1} \mathbf{a}')\\
&\ \hspace{5cm} \text{(by induction hypothesis and $\Omega$ being commutative)}\\
&=\ \mathbf{b} \succ_{\alpha} \mathbf{a}.
\end{align*}
Similarly, it can be proved that $\mathbf{a} \circ_{\alpha} \mathbf{b}=\mathbf{b} \circ_{\alpha} \mathbf{a}$. Hence $\shh(A)$ is a commutative $\Omega$-tridendriform algebra.
\end{proof}

Let $i:A \rightarrow \shh(A)$ be the natural inclusion. We obtain the following result:

\begin{theorem}\label{theouniv}
Let $(\Omega, \leftarrow, \rightarrow, \lhd, \rhd, \cdot, \ast)$ be a commutative ETS and let $(A, (\star_{\omega})_{\omega \in \Omega})$ be a commutative matching associative algebra. If $(B, (\prec_{\omega})_{\omega \in \Omega}, (\succ_{\omega})_{\omega \in \Omega}, (\circ_{\omega})_{\omega \in \Omega})$ is a commutative $\Omega$-tridendriform algebra and $\phi:(A,(\star_{\omega})_{\omega \in \Omega})\rightarrow (B,(\circ_{\omega})_{\omega \in \Omega})$ is a morphism of  matching associative algebras, then there exists a unique morphism $\Phi:\shh(A)\rightarrow B$ of $\Omega$-tridendriform algebras
 such that $\phi=\Phi \circ i$.
 \end{theorem}
In other terms, $\shh$ is the left adjoint functor of the forgetful functor from commutative $\Omega$-tridendriform algebras
 to commutative $\Omega$-associative algebras (which consists to forget $\prec$ and $\succ$).
As a consequence, the free commutative  $\Omega$-tridendriform algebra generated by $A$ is $\shh(A')$
where $A'$ is the free matching commutative algebra generated by $A$.

\begin{proof}
For $\mathbf{a} \in \shh(A)$, we define $\Phi(\mathbf{a})$ by induction on $\ell(\mathbf{a})$. If $\ell(\mathbf{a})=1$, then $\mathbf{a}=a_1 \in A$ and define $\Phi(\mathbf{a})= \phi(a_1)$. For the induction step of $\ell(\mathbf{a}) \geq 2$, suppose $\mathbf{a}=a_1 \ot_{\alpha_1} \mathbf{a}'$, then define
\begin{align*}
\Phi(\mathbf{a})=\Phi(a_1 \ot_{\alpha_1} \mathbf{a}')=\Phi(a_1 \prec_{\alpha_1} \mathbf{a}') :=\Phi(a_1) \prec_{\alpha_1} \Phi(\mathbf{a}').
\end{align*}

We can get that it is the unique way to extend $\phi$ to an $\Omega$-tridendriform algebra morphism $\Phi$ such that $\phi=\Phi \circ i$.\end{proof}

\subsection{From $\Omega$-tridendriform algebras to tridendriform algebras}

If $\Omega$ is an ETS, we consider the three maps
\begin{align*}
\varphi_\leftarrow&:\left\{\begin{array}{rcl}
\bfk\Omega^{\otimes 2}&\longrightarrow&\bfk\Omega^{\otimes 2}\\
\alpha\otimes \beta&\longmapsto&\alpha \leftarrow \beta\otimes \alpha \lhd \beta,
\end{array}\right.&
\varphi_\rightarrow&:\left\{\begin{array}{rcl}
\bfk\Omega^{\otimes 2}&\longrightarrow&\bfk\Omega^{\otimes 2}\\
\alpha\otimes \beta&\longmapsto&\alpha \rightarrow \beta\otimes \alpha \rhd \beta,
\end{array}\right.\\
\varphi_\ast&:\left\{\begin{array}{rcl}
\bfk\Omega^{\otimes 2}&\longrightarrow&\bfk\Omega^{\otimes 2}\\
\alpha\otimes \beta&\longmapsto&\alpha \cdot\beta\otimes \alpha \ast\beta.
\end{array}\right.&
\end{align*}

\begin{prop} \label{prop3.6}
Let $\Omega$ be an ETS and let $A$ be a vector space equipped with bilinear products $\prec_{\omega}, \succ_{\omega}, \circ_{\omega}$. We equip $\bfk \Omega \ot A$ with three bilinear products $\prec, \succ, \circ$ defined by
\begin{align*}
\alpha \ot x \prec \beta \ot y:&=\ \alpha \leftarrow \beta \ot x \prec_{\alpha \lhd \beta} y,\\
\alpha \ot x \succ \beta \ot y:&=\ \alpha \rightarrow \beta \ot x \succ_{\alpha \rhd \beta} y,\\
\alpha \ot x \circ \beta \ot y:&=\ \alpha \cdot \beta \ot x \circ_{\alpha \ast \beta} y,
\end{align*}
for all $\alpha, \beta \in \Omega$ and $x,y \in A$.
\begin{enumerate}
\item \label{it:aa} If $(A, (\prec_{\omega})_{\omega \in \Omega}, (\succ_{\omega})_{\omega \in \Omega}, (\circ_{\omega})_{\omega \in \Omega})$ is an $\Omega$-tridendriform algebra, then $(\bfk \Omega \ot A, \prec, \succ, \circ)$ is a tridendriform algebra.

\item \label{it:bb} If $\varphi_{\leftarrow}$, $\varphi_{\rightarrow}$ and $\varphi_{\ast}$ are surjective, then the converse implication is true.

\end{enumerate}
\end{prop}

\begin{proof}
\ref{it:aa} For $\alpha, \beta,\gamma \in \Omega$ and $a,b,c \in A$,
\begin{align*}
&\ (\alpha \ot a \prec \beta \ot b) \prec \gamma \ot c=(\alpha \leftarrow \beta \ot a \prec_{\alpha \lhd \beta} b) \prec \gamma \ot c\\
&=\ (\alpha \leftarrow \beta) \leftarrow \gamma \ot (a \prec_{\alpha \lhd \beta} b)\prec_{(\alpha \leftarrow \beta)\lhd \gamma} c\\
&=\ (\alpha \leftarrow \beta) \leftarrow \gamma \ot \left(\begin{array}{c}
a \prec_{(\alpha \lhd \beta) \rightarrow ((\alpha \leftarrow \beta) \lhd \gamma)}(b \succ_{(\alpha \lhd \beta) \rhd ((\alpha \leftarrow \beta) \lhd \gamma)} c)\\
+a \prec_{(\alpha \lhd \beta)\leftarrow ((\alpha \leftarrow \beta) \lhd \gamma)} (b \prec_{(\alpha \lhd \beta) \lhd ((\alpha \leftarrow \beta) \lhd \gamma)} c)\\
\ + a\prec_{(\alpha \lhd \beta) \cdot ((\alpha \leftarrow \beta) \lhd \gamma)} (b \circ_{(\alpha \lhd \beta)\ast ((\alpha \leftarrow \beta) \lhd \gamma)}c)
\end{array}\right) \\
&\quad\quad \text{(by $A$ being an $\Omega$-tridendriform algebra)}\\
&=\ \alpha \leftarrow (\beta \leftarrow \gamma) \ot a \prec_{\alpha \lhd (\beta \leftarrow \gamma)} (b \leftarrow_{\beta \lhd \gamma} c)+ \alpha \leftarrow (\beta \rightarrow \gamma) \ot a \prec_{\alpha \lhd (\beta \rightarrow \gamma)}(b \succ_{\beta \rhd \gamma} c)\\
&\ +\alpha \leftarrow (\beta \cdot \gamma) \ot a \prec_{\alpha \lhd (\beta \cdot \gamma)} (b \circ_{\beta \ast \gamma} c)
\quad\quad \text{(by $\Omega$ being an ETS)}\\
&=\ \alpha \ot a \prec(\beta \leftarrow \gamma \ot b \prec_{\beta \lhd \gamma} c+ \beta \rightarrow \gamma \ot b \succ_{\beta \rhd \gamma} c+ \beta \cdot \gamma \ot b \circ_{\beta \ast \gamma} c)\\
&=\ \alpha \ot a \prec(\beta \ot b \prec \gamma \ot c+ \beta \ot b \succ \gamma \ot c+ \beta \ot b \circ \gamma \ot c).
\end{align*}
The other equations can be proved in the same way. Hence $(\bfk \Omega \ot A, \prec, \succ, \circ)$ is a tridendriform algebra.\\

\ref{it:bb} For $\alpha, \beta, \gamma \in \Omega$ and $a,b,c \in A$, by $\Omega$ being an ETS and
\begin{align*}
(\alpha \ot a \prec \beta \ot b) \prec \gamma \ot c= \alpha \ot a \prec(\beta \ot b \prec \gamma \ot c+ \beta \ot b \succ \gamma \ot c+ \beta \ot b \circ \gamma \ot c),
\end{align*}
we get
\begin{align*}
(a \prec_{\alpha \lhd \beta} b)\prec_{(\alpha \leftarrow \beta)\lhd \gamma} c=a \prec_{\alpha \lhd (\beta \leftarrow \gamma)} (b \leftarrow_{\beta \lhd \gamma} c)+a \prec_{\alpha \lhd (\beta \rightarrow \gamma)}(b \succ_{\beta \rhd \gamma} c)+a \prec_{\alpha \lhd (\beta \cdot \gamma)} (b \circ_{\beta \ast \gamma} c)
\end{align*}

By hypothesis, the following map is surjective:
\begin{align*}
\phi'_\leftarrow&:\left\{\begin{array}{rcl}
\Omega^{2}&\longrightarrow& \Omega^{2}\\
(\alpha, \beta)&\longrightarrow& (\alpha \lhd \beta, \alpha \leftarrow \beta).
\end{array}\right.
\end{align*}
Hence, the following composition is surjective:
\begin{align*}
(\Id \ot \phi'_\leftarrow) \circ (\phi'_{\leftarrow} \ot \Id)&:\left\{\begin{array}{rcl}
\Omega^{3}&\longrightarrow&\Omega^{3}\\
(\alpha, \beta, \gamma) &\longrightarrow& (\alpha \lhd \beta, (\alpha \leftarrow \beta) \lhd \gamma,  (\alpha \leftarrow \beta) \leftarrow \gamma)
\end{array}\right.
\end{align*}
Let $(\alpha, \beta,\gamma), (\alpha',\beta',\gamma') \in \Omega^3$ such that
\begin{align*}
\alpha'&=\alpha \lhd \beta,&\beta'&=(\alpha \leftarrow \beta) \lhd \gamma,&
\gamma'&=(\alpha \leftarrow \beta) \leftarrow \gamma.
\end{align*}
Then
\begin{align*}
\ (a \prec_{\alpha'} b) \prec_{\beta'} c&=\ a \prec_{\alpha \lhd (\beta \leftarrow \gamma)} (b \leftarrow_{\beta \lhd \gamma} c)+a \prec_{\alpha \lhd (\beta \rightarrow \gamma)}(b \succ_{\beta \rhd \gamma} c)+a \prec_{\alpha \lhd (\beta \cdot \gamma)} (b \circ_{\beta \ast \gamma} c)\\
&=\ a\prec_{\alpha' \leftarrow \beta'} (b \prec_{\alpha' \lhd \beta'} c)+ a \prec_{\alpha' \rightarrow \beta'}(b \succ_{\alpha' \rhd \beta'} c)+ a \prec_{\alpha' \cdot \beta'}(b \circ_{\alpha' \ast \beta'} c).
\end{align*}
So Eq.~(\ref{eq:tri1}) holds. Eqs.~(\ref{eq:tri2})-(\ref{eq:tri7}) can be proved similarly.
\end{proof}

\begin{prop}\label{prop3.7}
Let $\Omega$ be an ETS.

$(1)$ The following assertions are equivalent:
\begin{enumerate}
\item \label{it:a11a} The tridendriform algebra $\bfk \Omega \ot \bfk \frat(X,\Omega)$ is generated by the elements $\omega \ot \stree {x}$, where $\omega \in \Omega$ and $x \in X$.

\item \label{it:a11b} The maps $\varphi_{\leftarrow}$, $\varphi_{\rightarrow}$ and $\varphi_{\ast}$ are surjective.

\end{enumerate}

$(2)$ The following assertions are equivalent:
\begin{enumerate}
\item \label{it:b11a} The tridendriform subalgebra of $\bfk \Omega \ot \bfk \frat(X, \Omega)$ generated by the elements $\omega \ot \stree {x}$, where $\omega \in \Omega$ and $x \in X$, is free.
\item \label{it:b11b} The maps  $\varphi_{\leftarrow}$, $\varphi_{\rightarrow}$ and $\varphi_{\ast}$ are injective.
\end{enumerate}

\end{prop}

\begin{proof}
Note that the tridendriform algebra $\bfk \Omega \ot \frat(X, \Omega)$ is graded, with for each $n \geq 1$,
\begin{align*}
(\bfk \Omega \ot \frat(X,\Omega))_n=\bfk \Omega \ot \frat_n(X, \Omega).
\end{align*}

(1) $\ref{it:a11a}\Longrightarrow \ref{it:a11b}$ As $\bfk \Omega \ot \frat(X, \Omega)$ is graded, by hypothesis, for any $\alpha \ot \XX{\xxr{-5}5
\node at (0.2,1.2) {\tiny $\beta$};
\xxhu00x \xxhu{-5}5y
} \in (\Omega \ot \frat(X,\Omega))_2$, where $\alpha, \beta \in \Omega$ and $x,y \in X$ , there are $\alpha_1 \ot \stree {y'}, \beta_1 \ot \stree {x'} \in ( \Omega \ot \frat(X,\Omega))_2$ and $p_{\alpha_1, \beta_1} \in \bfk$ such that
\begin{align*}
\ \alpha \ot \XX{\xxr{-5}5
\node at (0.2,1.2) {\tiny $\beta$};
\xxhu00x \xxhu{-5}5y
}
&=\ \sum\limits_{\alpha_1, \beta_1 \in \Omega} p_{\alpha_1, \beta_1} \alpha_1 \ot \stree {y'} \succ \beta_1 \ot \stree {x'}= \sum\limits_{\alpha_1 ,\beta_1 \in \Omega} p_{\alpha_1, \beta_1} \alpha_1 \rightarrow \beta_1 \ot \XX{\xxr{-5}5
\node at (0.2,1.5) {\tiny $\alpha_1 \rhd \beta_1$};
\xxhu00{x'} \xxhu{-5}5{y'}
}.
\end{align*}
Hence, there exists $(\alpha_1, \beta_1) \in \Omega^2$ such that $\alpha_1 \rightarrow \beta_1=\alpha$ and $\alpha_1 \rhd \beta_1 =\beta$. So $\varphi_{\rightarrow}$ is surjective. Similarly, $\varphi_{\leftarrow}$ and $\varphi_{\ast}$ are surjective.\\

$\ref{it:a11b} \Longrightarrow \ref{it:a11a}$ We prove that any $\alpha \ot T \in \bfk \Omega \ot \bfk \frat(X,\Omega)$, where $\alpha \in \Omega$ and $T \in \frat(X, \Omega)$, is generated by $\omega \ot \stree {x}$ by induction on the number $N$ of leaves of $T$. If $N=2$, then $T=\stree {x}$ for some $x \in X$ and it is obvious. Suppose $\alpha \ot T$ is generated by $\omega \ot \stree{x}$ for $N \leq p$, where $p \geq 2$ is a fixed integer. Consider the case of $N=p+1$. We consider the form of $T$ as follows.
\begin{enumerate}
\item If $T=\treeoo{\cdb o
\cdx[2]{o}{a2}{120}
\ocdx[2]{o}{a3}{60}{\prescript{\beta}{}{T_2}}{above}
\node at (90:\xch) {$x$};
}$, let $(\alpha_1, \beta_1 ) \in \Omega^2$ such that $\varphi_{\leftarrow}(\alpha_1, \beta_1)=(\alpha_1 \leftarrow \beta_1, \alpha_1 \lhd \beta_1)=(\alpha, \beta)$. Then
\begin{align*}
\alpha_1 \ot \stree x \prec \beta_1 \ot T_2=(\alpha_1 \leftarrow \beta_1) \ot \treeoo{\cdb o
\cdx[2]{o}{a2}{120}
\ocdx[2]{o}{a3}{60}{\prescript{\alpha_1 \lhd \beta_1}{}{T_2}}{above}
\node at (90:\xch) {$x$};
}=\alpha \ot T.
\end{align*}
Hence, by induction hypothesis, $\alpha \ot T$ is generated by $\omega \ot \stree x$.

\item If $T= \treeoo{\cdb o\cdx[2]{o}{a1}{160}
\ocdx[2]{o}{a2}{120}{\prescript{\beta}{}{T_2}}{above}
\ocdx[2]{o}{a3}{60}{T_{m}}{above}
\ocdx[2]{o}{a4}{20}{T_{m+1}}{right}
\node at (140:\xch) {$x_1$};
\node at (90:\xch) {$\cdots$};
\node at (40:\xch) {$x_m$};
}$ with $m \geq 2$, let $(\alpha_1,\beta_1) \in \Omega^2$ such that $\varphi_{\ast}(\alpha_1, \beta_1)=(\alpha_1 \ast \beta_1, \alpha_1 \cdot \beta_1)=(\beta, \alpha)$. Then
\begin{align*}
\alpha_1 \ot \stree {x} \circ \beta_1 \ot \treeoo{\cdb o
\ocdx[2]{o}{a2}{135}{T_2}{above}
\ocdx[2]{o}{a3}{90}{T_{m}}{above}
\ocdx[2]{o}{a4}{45}{T_{m+1}}{right}
\node at (110:\xch) {$\cdots$};
\node at (70:\xch) {$x_m$};
}=\alpha_1 \cdot \beta_1 \ot  \treeoo{\cdb o\cdx[2]{o}{a1}{160}
\ocdx[2]{o}{a2}{120}{\prescript{\alpha_1 \ast \beta_1}{}{T_2}}{above}
\ocdx[2]{o}{a3}{60}{T_{m}}{above}
\ocdx[2]{o}{a4}{20}{T_{m+1}}{right}
\node at (140:\xch) {$x_1$};
\node at (90:\xch) {$\cdots$};
\node at (40:\xch) {$x_m$};
} =\alpha \ot T.
\end{align*}
Hence, by induction hypothesis, $\alpha \ot T$ is generated by $\omega \ot \stree x$.

\item If $T=\treeoo{\cdb o\ocdx[2]{o}{a1}{160}{T_1^{\beta}}{left}
\ocdx[2]{o}{a2}{120}{T_2}{above}
\ocdx[2]{o}{a3}{60}{T_m}{above}
\ocdx[2]{o}{a4}{20}{T_{m+1}}{right}
\node at (140:\xch) {$x_1$};
\node at (90:\xch) {$\cdots$};
\node at (40:\xch) {$x_m$};
}$ with $T_1 \neq |$, let $(\alpha_1, \beta_1) \in \Omega^2$ such that $\varphi_{\rightarrow}(\alpha_1,\beta_1)=(\alpha_1 \rightarrow \beta_1, \alpha_1 \rhd \beta_1)=(\alpha,\beta)$. Then
\begin{align*}
\alpha_1 \ot T_1 \succ \beta_1 \ot \treeoo{\cdb o\cdx[2]{o}{a1}{160}
\ocdx[2]{o}{a2}{120}{T_2}{above}
\ocdx[2]{o}{a3}{60}{T_{m}}{above}
\ocdx[2]{o}{a4}{20}{T_{m+1}}{right}
\node at (140:\xch) {$x_1$};
\node at (90:\xch) {$\cdots$};
\node at (40:\xch) {$x_m$};
}= (\alpha_1 \rightarrow \beta_1) \ot \treeoo{\cdb o\ocdx[2]{o}{a1}{160}{T_1^{\alpha_1 \rhd \beta_1}}{left}
\ocdx[2]{o}{a2}{120}{T_2}{above}
\ocdx[2]{o}{a3}{60}{T_m}{above}
\ocdx[2]{o}{a4}{20}{T_{m+1}}{right}
\node at (140:\xch) {$x_1$};
\node at (90:\xch) {$\cdots$};
\node at (40:\xch) {$x_m$};
} =\alpha \ot T.
\end{align*}
Hence, by induction hypothesis, $\alpha \ot T$ is generated by $\omega \ot \stree x$.\\
\end{enumerate}

(2) $\ref{it:b11a} \Longrightarrow \ref{it:b11b}$ Denote by $A$ the tridendriform subalgebra of $\bfk \Omega \ot \bfk \frat(X,\Omega)$ generated by elements $\omega \ot \stree x$. Let $(\alpha, \beta),(\alpha',\beta') \in \Omega^2$ such that $\varphi_{\leftarrow}(\alpha, \beta)=\varphi_{\leftarrow}(\alpha',\beta')$. Then
\begin{align*}
\alpha \ot \stree x \prec \beta \ot \stree y=\alpha \leftarrow \beta \ot \XX{\xxl55
\node at (-0.1,1.25) {\tiny $\alpha \lhd \beta$};
\xxhu00x \xxhu55y
}= \alpha' \leftarrow \beta' \ot \XX{\xxl55
\node at (-0.1,1.25) {\tiny $\alpha' \lhd \beta'$};
\xxhu00{x} \xxhu55{y}
}=\alpha' \ot \stree x \prec \beta' \ot \stree y.
\end{align*}
By the freeness of $A$, $(\alpha, \beta)=(\alpha',\beta')$ and so $\varphi_{\leftarrow}$ is injective. The maps $\varphi_{\rightarrow}$
and $\varphi_{\ast}$ can be proved to be injective similarly.\\

$\ref{it:b11b} \Longrightarrow \ref{it:b11a}$ Let $\mathrm{TDend}(\Omega)$ be the free tridendriform algebra generated by $\Omega \ot X$. As a vector space, it is generated by Schr\"oder trees which angles are decorated by $\Omega \ot X$. Let $\Phi: \mathrm{TDend}(\Omega) \rightarrow \bfk \Omega \ot \bfk \frat(X, \Omega)$ be the unique tridendriform algebra sending $\stree {\alpha \ot x}$ to $\alpha \ot \stree x$. We prove that $\Phi$ is injective, i.e.
\begin{align*}
\Phi(T)=\Phi(T') \Longrightarrow T=T'
\end{align*}
by induction on the number $N$ of leaves of $T$. By the construction of $\Phi$, if $\Phi(T)=\Phi(T')$, then $T,T'$ are of the same form. If $N=2$, then $T=\stree {\alpha \ot x}$ for some $\alpha \in \Omega$ and $x \in X$ and obviously $T'=T$.  Suppose $\Phi$ is injective for all $T$ with $N \leq p$, where $p$ is a fixed integer. Consider the case of $N=p+1$. If
\begin{align*}
T=\treeoo{\cdb o
\cdx[2]{o}{a2}{120}
\ocdx[2]{o}{a3}{60}{T_2}{above}
\node at (90:\xch) {$\alpha \ot x$};
} \,\, T'=\treeoo{\cdb o
\cdx[2]{o}{a2}{120}
\ocdx[2]{o}{a3}{60}{T'_2}{above}
\node at (90:\xch) {$\alpha' \ot x'$};
}
\end{align*}
and assume $\Phi(T_2)=\beta \ot U_2, \Phi(T'_2)= \beta' \ot U'_2$. Then
\begin{align*}
\Phi(T)&=\  \Phi\left(\stree {\alpha \ot x} \prec T_2\right)= \Phi\left(\stree {\alpha \ot x}\right) \prec \Phi(T_2)=\alpha \leftarrow \beta \ot  \treeoo{\cdb o
\cdx[2]{o}{a2}{120}
\ocdx[2]{o}{a3}{60}{\prescript{\alpha \lhd \beta}{}{T_2}}{above}
\node at (90:\xch) {$x$};
} \\
\Phi(T')&=\ \Phi\left(\stree {\alpha' \ot x'} \prec T'_2\right)=\Phi\left(\stree {\alpha' \ot x}\right) \prec \Phi(T'_2)=\alpha' \leftarrow \beta' \ot \treeoo{\cdb o
\cdx[2]{o}{a2}{120}
\ocdx[2]{o}{a3}{60}{\prescript{\alpha' \lhd \beta'}{}{T_2}}{above}
\node at (90:\xch) {$x$};
}.
\end{align*}
Since $\Phi(T)=\Phi(T')$ and $\varphi_{\leftarrow}$ is injective, $\alpha=\alpha',x=x'$ and $\Phi(T_2)=\Phi(T'_2)$. Hence by induction hypothesis, $T=T'$. For other forms of $T,T'$, the injectivity of $\Phi$ is proved similarly. Hence $A$ is isomorphic to the free tridendriform algebra $\mathrm{TDend}(\Omega)$ and so A is free.
\end{proof}

\section{Operad of $\Omega$-tridendriform algebras}

Denote by $\mathcal{P}_{\Omega}$ the (nonsymmetric) operad of $\Omega$-tridendriform algebras. It is generated by $\prec_{\alpha}$, $\circ_{\omega}$ and $\succ_{\omega} \in \mathcal{P}_{\Omega}(2)$ with $\alpha \in \Omega$ and the relations:
\begin{align*}
\prec_{\beta} \circ (\prec_{\alpha}, I)&=\ \prec_{\alpha \rightarrow \beta}\circ (I, \succ_{\alpha \rhd \beta})+ \prec_{\alpha \leftarrow \beta} \circ (I, \prec_{\alpha \lhd \beta})+\prec_{\alpha \cdot \beta} \circ (I, \circ_{\alpha \ast \beta}),\\
\prec_{\beta} \circ (\succ_{\alpha}, I)&=\ \succ_{\alpha} \circ (I, \prec_{\beta}),\\
\succ_{\alpha} \circ (I, \succ_{\beta})&=\ \succ_{\alpha \rightarrow \beta} \circ (\succ_{\alpha \rhd \beta}, I)+ \succ_{\alpha \leftarrow \beta} (\prec_{\alpha \lhd \beta}, I)+ \succ_{\alpha \cdot \beta} \circ  (\circ_{\alpha \ast \beta},I),\\
\circ_{\beta} \circ (\succ_{\alpha},I)&=\ \succ_{\alpha} \circ (I, \circ_{\beta}),\\
\circ_{\beta} \circ (\prec_{\alpha}, I)&=\ \circ_{\beta} \circ (I, \succ_{\alpha}),\\
\prec_{\beta} \circ (\circ_{\alpha}, I)&=\ \circ_{\alpha} \circ (I, \prec_{\beta}),\\
\circ_{\beta} \circ (\circ_{\alpha}, I)&=\ \circ_{\alpha} \circ (I, \circ_{\beta}),
\end{align*}
for all $\alpha, \beta \in \Omega$.

As in \cite[Proposition~21]{Foi20}, we obtain the following result:

\begin{prop}
Suppose $m \in \mathcal{P}_{\Omega}(2)$ is of the form
\begin{align*}
m=\sum \limits_{\alpha \in \Omega} a_{\alpha} \prec_{\alpha}+ \sum \limits_{\alpha \in \Omega} b_{\alpha} \circ_{\alpha} + \sum \limits_{\alpha \in \Omega} c_{\alpha} \succ_{\alpha},
\end{align*}
where $a_{\alpha}, b_{\alpha}, c_{\alpha} \in \bfk$. Then $m \circ (I,m)=m \circ (m,I)$ if and only if for any $\alpha, \beta \in \Omega$,
\begin{align*}
a_{\alpha} a_{\beta}&=\ \sum \limits_{\varphi_{\leftarrow} (\alpha',\beta')=(\alpha,\beta)} a_{\alpha'}a_{\beta'}, && a_{\alpha}b_{\beta}=\sum \limits_{\varphi_{\ast}(\alpha',\beta')=(\alpha,\beta)} a_{\alpha'}a_{\beta'},\\
a_{\alpha} c_{\beta}&=\ \sum \limits_{\varphi_{\rightarrow} (\alpha',\beta')=(\alpha,\beta)} a_{\alpha'}a_{\beta'},&& b_{\alpha}c_{\beta}=b_{\alpha}a_{\beta},\\
c_{\alpha}a_{\beta}&=\ \sum \limits_{\varphi_{\leftarrow}(\alpha', \beta')=(\alpha,\beta)}c_{\alpha'} c_{\beta'}, && c_{\alpha}b_{\beta}= \sum \limits_{\varphi_{\ast}(\alpha',\beta')=(\alpha, \beta)} c_{\alpha'}c_{\beta'},\\
c_{\alpha}c_{\beta}&=\ \sum \limits_{\varphi_{\rightarrow}(\alpha',\beta')=(\alpha,\beta)} c_{\alpha'}c_{\beta'}.
\end{align*}
\end{prop}

\begin{proof}
By the relations of the operad of $\Omega$-tridendriform algebras,
\begin{align*}
&\ m \circ (I, m)=(a_{\alpha}\prec_{\alpha}+b_{\alpha} \circ_{\alpha}+ c_{\alpha} \succ_{\alpha}) \circ (I, a_{\beta} \prec_{\beta}+ b_{\beta} \circ_{\beta}+ c_{\beta} \succ_{\beta})\\
&=\ \sum \limits_{\alpha, \beta}a_{\alpha}a_{\beta} \prec_{\alpha} \circ (I, \prec_{\beta})+ \sum \limits_{\alpha, \beta}a_{\alpha}b_{\beta} \prec_{\alpha} \circ (I, \circ_{\beta})+\sum \limits_{\alpha, \beta} \prec_{\alpha} \circ (I, \succ_{\beta})+ \sum \limits_{\alpha, \beta} b_{\alpha} a_{\beta} \circ_{\alpha} \circ (I, \prec_{\beta})\\
&\ +\sum \limits_{\alpha, \beta} b_{\alpha}b_{\beta} \circ_{\alpha} \circ (I, \circ_{\beta})+\sum \limits_{\alpha, \beta} b_{\alpha}c_{\beta} \circ_{\alpha} \circ (I, \succ_{\beta})+\sum \limits_{\alpha, \beta} c_{\alpha}a_{\beta} \succ_{\alpha} \circ(I, \prec_{\beta})+\sum \limits_{\alpha, \beta} c_{\alpha}b_{\beta} \succ_{\alpha} \circ (I, \circ_{\beta})\\
&\ +\sum \limits_{\alpha, \beta} c_{\alpha} c_{\beta} \succ_{\alpha \rightarrow \beta} \circ (\succ_{\alpha \rhd \beta},I)+\sum \limits_{\alpha, \beta} c_{\alpha} c_{\beta} \succ_{\alpha \leftarrow \beta} (\prec_{\alpha \lhd \beta},I)+\sum \limits_{\alpha, \beta} c_{\alpha} c_{\beta} \succ_{\alpha \cdot \beta} \circ (\circ_{\alpha \ast \beta}, I)
\end{align*}
and
\begin{align*}
&\ m \circ (m,I)=(a_{\alpha} \prec_{\alpha}+b_{\alpha} \circ_{\alpha}+ c_{\alpha} \succ_{\alpha}) \circ (a_{\alpha} \prec_{\alpha}+b_{\alpha} \circ_{\alpha}+ c_{\alpha} \succ_{\alpha},I)\\
&=\ \sum \limits_{\alpha, \beta} a_{\alpha} a_{\beta} \prec_{\alpha \rightarrow \beta}\circ (I,\succ_{\alpha \rhd \beta})+ \sum \limits_{\alpha, \beta} a_{\alpha}a_{\beta} \prec_{\alpha \leftarrow \beta}(I, \prec_{\alpha \lhd \beta})+\sum \limits_{\alpha, \beta}a_{\alpha}a_{\beta} \prec_{\alpha \cdot \beta} \circ (I, \circ_{\alpha \ast \beta})\\
&\ +\sum \limits_{\alpha, \beta}a_{\alpha}b_{\beta} \prec_{\alpha} \circ (\circ_{\beta},I)+\sum \limits_{\alpha, \beta} a_{\alpha} c_{\beta} \prec_{\alpha} \circ (\succ_{\beta},I)+\sum \limits_{\alpha, \beta} b_{\alpha}a_{\beta} \circ_{\alpha} \circ (\prec_{\beta},I)+\sum \limits_{\alpha, \beta} b_{\alpha}b_{\beta} \circ_{\alpha} \circ (\circ_{\beta},I)\\
&\ +\sum \limits_{\alpha, \beta} b_{\alpha}c_{\beta} \circ_{\alpha} \circ (\succ_{\beta},I)+ \sum \limits_{\alpha, \beta} c_{\alpha}a_{\beta} \succ_{\alpha} \circ (\prec_{\beta},I)+\sum \limits_{\alpha, \beta} c_{\alpha}b_{\beta} \succ_{\alpha} \circ (\circ_{\beta},I)+\sum \limits_{\alpha, \beta} c_{\alpha}c_{\beta} \succ_{\alpha} \circ (\succ_{\beta},I).
\end{align*}
Hence $m \circ (I,m)=m \circ (m,I)$ if and only if the above equations hold.
\end{proof}

\begin{remark}
These conditions can be reformulated as follows. We extend $\varphi_\leftarrow$, $\varphi_\rightarrow$ and $\varphi_\ast$
as linear endomorphisms $\bfk\Omega^{\otimes}$. We then consider the three elements of $\bfk\Omega$:
\begin{align*}
a&=\sum_{\alpha \in \Omega} a_\alpha\alpha,&
b&=\sum_{\alpha \in \Omega} b_\alpha\alpha,&
c&=\sum_{\alpha \in \Omega} c_\alpha\alpha.
\end{align*}
Then $m$ is associative if, and only if
\begin{align*}
\varphi_\leftarrow(a\otimes a)&=a\otimes a,&\varphi_\ast(a\otimes a)&=a\otimes b,\\
\varphi_\rightarrow(a\otimes a)&=a\otimes c,&\mbox{$b=0$}&\mbox{ or $a=c$},\\
\varphi_\leftarrow(c\otimes c)&=c\otimes a,&\varphi_\ast(c\otimes c)&=c\otimes b,\\
\varphi_\rightarrow(c\otimes c)&=c\otimes c,
\end{align*}
which is equivalent to
\begin{align*}
&&\left\{\begin{array}{rcl}
b&=&0,\\
\varphi_\leftarrow(a\otimes a)&=&a\otimes a,\\
\varphi_\rightarrow(a\otimes a)&=&a\otimes c,\\
\varphi_\ast(a\otimes a)&=&0,\\
\varphi_\leftarrow(c\otimes c)&=&c\otimes a,\\
\varphi_\rightarrow(c\otimes c)&=&c\otimes c,\\
\varphi_\ast(c\otimes c)&=&0,
\end{array}\right.&&\mbox{or}&&\left\{\begin{array}{rcl}
c&=&a,\\
\varphi_\leftarrow(a\otimes a)&=&a\otimes a,\\
\varphi_\rightarrow(a\otimes a)&=&a\otimes a,\\
\varphi_\ast(a\otimes a)&=&a\otimes b.
\end{array}\right.\\
\end{align*}

\end{remark}

If $\Omega$ is finite, the operad $\mathcal{P}_{\Omega}$ is a finite generated quadratic operad. By computation, we obtain the Koszul dual of $\mathcal{P}_{\Omega}$ as follows.

\begin{prop}\label{propKoszul}
Let $\Omega$ be a finite ETS. The Koszul dual $\mathcal{P}^{!}_{\Omega}$ of $\mathcal{P}_{\Omega}$ is generated by $\dashv_{\alpha}, \perp_{\alpha}, \vdash_{\alpha}$ with $\alpha \in \Omega$ and the relations
\begin{align*}
\dashv_\alpha \circ(I,\dashv_\beta)&=
\sum_{\substack{(\gamma,\delta)\in \Omega^2,\\ \gamma\leftarrow\delta=\alpha,\\
\gamma \lhd \delta=\beta}}\dashv_\delta\circ (\dashv_\gamma,I),&
\dashv_\alpha \circ(I,\vdash_\beta)&=
\sum_{\substack{(\gamma,\delta)\in \Omega^2,\\ \gamma\rightarrow\delta=\alpha,\\
\gamma \rhd \delta=\beta}}\dashv_\delta\circ (\dashv_\gamma,I),\\
\dashv_\alpha \circ(I,\perp_\beta)&=
\sum_{\substack{(\gamma,\delta)\in \Omega^2,\\ \gamma\cdot\delta=\alpha,\\
\gamma \ast \delta=\beta}}\dashv_\delta\circ (\dashv_\gamma,I),\\
\vdash_\beta \circ(\vdash_\alpha,I)&=
\sum_{\substack{(\gamma,\delta)\in \Omega^2,\\ \gamma\rightarrow\delta=\alpha,\\
\gamma \rhd \delta=\beta}}\vdash_\gamma\circ (I,\vdash_\delta),&
\vdash_\beta \circ(\dashv_\alpha,I)&=
\sum_{\substack{(\gamma,\delta)\in \Omega^2,\\ \gamma\leftarrow\delta=\alpha,\\
\gamma \lhd \delta=\beta}}\vdash_\gamma\circ (I,\vdash_\delta),\\
\vdash_\beta \circ(\perp_\alpha,I)&=
\sum_{\substack{(\gamma,\delta)\in \Omega^2,\\ \gamma\cdot\delta=\alpha,\\
\gamma \ast \delta=\beta}}\vdash_\gamma\circ (I,\vdash_\delta),\\
\dashv_{\beta} \circ (\vdash_{\alpha}, I)&=\ \vdash_{\alpha} \circ (I, \dashv_{\beta}),&
\perp_{\beta} \circ (\vdash_{\alpha}, I)&=\ \vdash_{\alpha} \circ (I, \perp_{\beta}),\\
\perp_{\beta} \circ (\dashv_{\alpha},I)&=\ \perp_{\beta} \circ (I, \vdash_{\alpha}),&
\dashv_{\beta} \circ (\perp_{\alpha},I)&=\ \perp_{\alpha} \circ (I, \dashv_{\beta}),\\
\perp_{\beta} \circ (\perp_{\alpha}, I)&=\ \perp_{\alpha} \circ (I, \perp_{\beta}),
\end{align*}
for all $\alpha, \beta \in \Omega$.
\end{prop}

In particular, if $|\Omega|=1$, we recover the definition of triassociative algebras, which operad is the Koszul
dual of the operad of tridendriform algebras \cite{LRtridend}.

\smallskip

\noindent {\bf Acknowledgments}:
The authors acknowledge support from the grant ANR-20-CE40-0007
\emph{Combinatoire Alg\'ebrique, R\'esurgence, Probabilit\'es Libres et Op\'erades}. The second author is supported by the National Natural Science Foundation of China (Grant No. 11771191 and 12101316) and he is also supported by China Scholarship Council to visit ULCO.
\medskip

\bibliographystyle{amsplain}
\bibliography{biblio}

\end{document}

\end{document}